\documentclass{article}

\AtEndDocument{\bigskip{\footnotesize%
Y.~Harpaz: \textsc{Radboud Universiteit Nijmegen, Institute for
Mathematics, Astrophysics, and Particle Physics, Heyendaalseweg 135, 6525
AJ Nijmegen, The Netherlands} \par  
  \textit{E-mail address:} \texttt{harpazy@gmail.com} \par

  \addvspace{\medskipamount}
M.~Prasma: \textsc{Radboud Universiteit Nijmegen, Institute for
Mathematics, Astrophysics, and Particle Physics, Heyendaalseweg 135, 6525
AJ Nijmegen, The Netherlands} \par
  \textit{E-mail address:}  \texttt{mtnprsm@gmail.com}
}}

\usepackage[centertags]{amsmath}
\usepackage{hyperref}
\usepackage{amsfonts}
\usepackage{MnSymbol}
\usepackage{amsthm}
\usepackage{newlfont}
\usepackage{amscd}
\usepackage{amsmath}
\usepackage{enumerate}
\usepackage[all,2cell]{xy}
\UseAllTwocells
\input xy
\xyoption{2cell}
\xyoption{all}
\usepackage{verbatim}
\usepackage{eucal}

\numberwithin{equation}{subsection}

\newcommand{\Cof}{\mathcal{C}of}
\newcommand{\Fib}{\mathcal{F}ib}

\newcommand{\GG}{\mathfrak{G}}
\newcommand{\II}{\mathbb{I}}

\def\B{\mathbb{B}}
\def\C{\mathcal{C}}
\def\D{\mathcal{D}}
\def\G{\mathcal{G}}

\def\F{\mathcal{F}}
\def\SS{\mathcal{S}}
\def\M{\mathcal{M}}
\def\N{\mathcal{N}}
\def\W{\mathcal{W}}
\def\I{\mathcal{I}}
\def\SS{\mathcal{S}}

\def\P{\mathcal{P}}
\def\J{\mathcal{J}}
\def\R{\mathcal{R}}
\renewcommand{\L}{\mathcal{L}}
\def\U{\mathcal{U}}
\def\V{\mathcal{V}}
\def\W{\mathcal{W}}
\def\O{\mathcal{O}}

\def\alp{{\alpha}}

\def\gam{{\gamma}}
\def\del{\delta}
\def\eps{{\varepsilon}}

\def\sig{{\sigma}}

\def\Del{{\Delta}}
\def\Sig{{\Sigma}}

\def\vphi{\varphi}

\newcommand{\HSwarrow}{\kern0.05ex\vcenter{\hbox{\Huge\ensuremath{\Swarrow}}}\kern0.05ex}
\newcommand{\hSwarrow}{\kern0.05ex\vcenter{\hbox{\huge\ensuremath{\Swarrow}}}\kern0.05ex}
\newcommand{\LLSwarrow}{\kern0.05ex\vcenter{\hbox{\LARGE\ensuremath{\Swarrow}}}\kern0.05ex}
\newcommand{\LSwarrow}{\kern0.05ex\vcenter{\hbox{\Large\ensuremath{\Swarrow}}}\kern0.05ex}

\newcommand{\HSearrow}{\kern0.05ex\vcenter{\hbox{\Huge\ensuremath{\Searrow}}}\kern0.05ex}
\newcommand{\hSearrow}{\kern0.05ex\vcenter{\hbox{\huge\ensuremath{\Searrow}}}\kern0.05ex}
\newcommand{\LLSearrow}{\kern0.05ex\vcenter{\hbox{\LARGE\ensuremath{\Searrow}}}\kern0.05ex}
\newcommand{\LSearrow}{\kern0.05ex\vcenter{\hbox{\Large\ensuremath{\Searrow}}}\kern0.05ex}

\newcommand{\HDownarrow}{\kern0.05ex\vcenter{\hbox{\Huge\ensuremath{\Downarrow}}}\kern0.05ex}
\newcommand{\hDownarrow}{\kern0.05ex\vcenter{\hbox{\huge\ensuremath{\Downarrow}}}\kern0.05ex}
\newcommand{\LLDownarrow}{\kern0.05ex\vcenter{\hbox{\LARGE\ensuremath{\Downarrow}}}\kern0.05ex}
\newcommand{\LDownarrow}{\kern0.05ex\vcenter{\hbox{\Large\ensuremath{\Downarrow}}}\kern0.05ex}

\newcommand{\HUparrow}{\kern0.05ex\vcenter{\hbox{\Huge\ensuremath{\Uparrow}}}\kern0.05ex}
\newcommand{\hUparrow}{\kern0.05ex\vcenter{\hbox{\huge\ensuremath{\Uparrow}}}\kern0.05ex}
\newcommand{\LLUparrow}{\kern0.05ex\vcenter{\hbox{\LARGE\ensuremath{\Uparrow}}}\kern0.05ex}
\newcommand{\LUparrow}{\kern0.05ex\vcenter{\hbox{\Large\ensuremath{\Uparrow}}}\kern0.05ex}

\newtheorem{thm}{Theorem}[subsection]
\newtheorem{cor}[thm]{Corollary}
\newtheorem{lem}[thm]{Lemma}
\newtheorem{pro}[thm]{Proposition}

\theoremstyle{definition}
\newtheorem{define}[thm]{Definition}

\newtheorem{example}[thm]{Example}
\newtheorem{defn}[thm]{Definition}

\newtheorem{notn}[thm]{Notation}
\newtheorem{obs}[thm]{Observation}

\theoremstyle{remark}
\newtheorem{rem}[thm]{Remark}

\DeclareMathOperator{\colim}{colim}

\DeclareMathOperator{\Id}{Id}
\DeclareFontFamily{OT1}{pzc}{}
\DeclareFontShape{OT1}{pzc}{m}{it}{<-> s * [1.10] pzcmi7t}{}
\DeclareMathAlphabet{\mathpzc}{OT1}{pzc}{m}{it}
\DeclareMathOperator{\cof}{cof}
\DeclareMathOperator{\fib}{fib}
\DeclareMathOperator{\ad}{ad}
\DeclareMathOperator{\res}{res}
\DeclareMathOperator{\Alg}{Alg}
\DeclareMathOperator{\LMod}{LMod}
\DeclareMathOperator{\RMod}{RMod}
\DeclareMathOperator{\ModCat}{ModCat}
\DeclareMathOperator{\sGr}{sGr}

\DeclareMathOperator{\op}{op}

\DeclareMathOperator{\df}{def}
\DeclareMathOperator{\Set}{Set}
\DeclareMathOperator{\RelCat}{RelCat}

\DeclareMathOperator{\Cat}{Cat}

\DeclareMathOperator{\CAlg}{CAlg}
\DeclareMathOperator{\Mod}{Mod}
\DeclareMathOperator{\Ch}{Ch}
\DeclareMathOperator{\Sym}{Sym}
\DeclareMathOperator{\Fun}{Fun}
\DeclareMathOperator{\AdjCat}{AdjCat}
\DeclareMathOperator{\Obj}{Obj}
\DeclareMathOperator{\BiFib}{biCar}
\DeclareMathOperator{\ModFib}{ModFib}
\DeclareMathOperator{\PR}{PR}
\DeclareMathOperator{\GFib}{Car}
\DeclareMathOperator{\OpFib}{coCar}

\DeclareMathOperator{\Ob}{Ob}

\def\x{\overset}

\def\Hom{\textrm{Hom}}

\newcommand{\tgpd}{\kern0.05ex\vcenter{\hbox{\footnotesize\ensuremath{2}}}\kern0.05ex\mathcal{G}pd} 

\def\rar{\rightarrow}

\def\lrar{\longrightarrow}

\def\hrar{\hookrightarrow}

\def\thrar{\twoheadrightarrow}

\def\ovl{\overline}

\newcommand\ackname{\textbf{Acknowledgements}}
\if@titlepage
  \newenvironment{acknowledgements}{
      \titlepage
      \null\vfil
      \@beginparpenalty\@lowpenalty
      \begin{center}
        \bfseries \ackname\
        \@endparpenalty\@M
      \end{center}}
     {\par\vfil\null\endtitlepage}
\else
  
\fi

\title{The Grothendieck construction for model categories}

\author{Yonatan Harpaz \;\;\;\; Matan Prasma}

\date{}
\begin{document}
\maketitle

\begin{abstract}
The Grothendieck construction is a classical correspondence between diagrams of categories and coCartesian fibrations over the indexing category. In this paper we consider the analogous correspondence in the setting of model categories. As a main result, we establish an equivalence between suitable diagrams of model categories indexed by $\M$ and a new notion of \textbf{model fibrations} over $\M$. When $\M$ is a model category, our construction endows the Grothendieck construction with a model structure which gives a presentation of Lurie's $\infty$-categorical Grothendieck construction and enjoys several good formal properties. We apply our construction to various examples, yielding model structures on strict and weak group actions and on modules over algebra objects in suitable monoidal model categories.

\end{abstract}

\tableofcontents

\section{Introduction}

In the $2$-category of categories, the lax colimit of a functor $\F: \C \lrar \Cat$ is represented by its \textbf{Grothendieck construction} $\int_\C \F$. Furthermore, the classical Grothendieck correspondence asserts that the association $\F \mapsto \int_\C \F$ gives rise to an equivalence of $(2,1)$-categories between functors $\C \lrar \Cat$ and coCartesian fibrations of categories $\D \lrar \C$. 

According to~\cite[\S I.5]{Ieke}, this construction was first used for diagrams of sets by Yoneda. It was later developed in full generality by Grothendieck in~\cite{SGA1} and became a key tool in studying categories which ``vary in families''. A prominent example is Grothendieck's original application to the study of categories of (quasi-)coherent sheaves which vary over the category of schemes (see~\cite{SGA1}).

In the world of $\infty$-categories, Lurie's straightening and unstraightening functors give the analogous correspondence. In this setting, the role of coCartesian fibrations is even more prominent, as it enables one to study homotopy coherent diagrams of $\infty$-categories without having to explicitly describe the vast net of coherences. One can also reinterpret in these terms various classical results, such as the classification theory of vector bundles. Recent work of Gepner, Haugseng and Nikolaus (\cite[Theorem 1.1]{GHN}) shows that this $\infty$-categorical Grothendieck construction is indeed a model for the $\infty$-categorical lax colimit. 

From this point of view, it is natural to consider the case where one is given a diagram $\F: \C \lrar \ModCat$ of \textbf{model categories}, and see if it is possible to analyse $\int_\C \F$ in model categorical terms, in a way that reflects the underlying diagram of $\infty$-categories. We note that in general, given a diagram of model categories $\F: \C \lrar \ModCat$, the Grothendieck construction $\int_\C\F$ need not even be bicomplete, and so one should not expect $\int_\C\F$ to carry any model structure. In this paper we will consider two possible solutions to this difficulty. One solution will be to consider the case where the indexing category is itself a model category $\M$. In such a setting, it is natural to assume that the diagram $\F:\M\lrar \ModCat$ is \textbf{relative} i.e. sends weak equivalences in $\M$ to Quillen equivalences. We will furthermore assume that $\F$ satisfied one additional property, which we refer to as $\F$ being ``proper" (see Definition~\ref{d:proper}). We will then show that under these assumptions, the category $\int_\M \F$ can be endowed with a natural model structure (see Theorem~\ref{model structure}), extending the model structures of the fibers, such that the canonical projection
$$ \int_\M \F \lrar \M $$
is both a left and a right Quillen functor. In \S\ref{ss:infinity-cat pov} we will prove that the map of the underlying $\infty$-categories induced by our construction coincides with that obtained from Lurie's $\infty$-categorical Grothendieck construction. By~\cite{GHN} this implies that our integral model category is also a model for the lax colimit of the corresponding diagram of $\infty$-categories.

This universal property above suggests that the integral model structure should be invariant (up to Quillen equivalence) under replacing $(\M,\F)$ with a suitably equivalent pair. We will prove in~\ref{s:invariance} that this is indeed the case (see Theorem~\ref{qa}).

A second approach taken up in this paper is to generalize the notion of a model category to a \textbf{relative setting}, in which case $\M$ need only carry three distinguished classed of maps (see \S\ref{s:model-fib}). Our end result is an \textbf{equivalence of $(2,1)$-categories} between proper relative diagrams $\F:\M \lrar \ModCat$ (in the formal sense) and a suitable notion of \textbf{model fibrations} $\N \lrar \M$ (see Theorem~\ref{t:equiv}). We consider this result as the model categorical analogue of the classical \textbf{Grothendieck's correspondence}. In particular, this result gives a substantial justification to our assumptions that $\F$ is proper and relative. We believe that the notions of model fibrations and relative model categories are interesting in their own right and may prove useful, for example in setting up a ``model category of model categories".

Theorem~\ref{model structure} seems to be widely applicable. We will demonstrate this in \S\ref{s:examples} by considering several classes of examples, ranging from slice and coslice categories to various algebras and their module categories. We will also consider an example for the invariance mentioned above in the case of strict versus weak group actions.

\subsection{Relation to other works}
Model structures on Grothendieck constructions were studied before, with the example of the fibred category of enriched categories as a main application. The first of these was in~\cite{Ro}, which was later corrected by~\cite{St}. In~\cite{St}, a model structure on $\int_\M \F$ is constructed under rather restrictive conditions on $\F$ and without the assumption that $\F$ is relative. As a result, the construction in~\cite{St} cannot be applied to most of the examples described in this paper (e.g. the slice construction, see \S\ref{ss:slice}), and does not enjoy the invariance property established in Theorem~\ref{qa}. Note that when $\F$ is relative the conditions of~\cite[Theorem $2.3$]{St} are strictly stronger than the assumptions of Theorem~\ref{model structure}. Moreover, when these conditions are satisfied, the two model structures agree.

We remark that the Grothendieck construction has also appeared in works on \textbf{lax limits} of model categories, where one is interested in constructing a model structure on the category of \textbf{sections} $\C \lrar \int_\C\F$. This was first worked out in~\cite{HS} when $\C$ is a Reedy category. It was later generalized by~\cite{Bar} to an arbitrary category $\C$ under suitable assumptions on $\F(c)$ for $c \in \C$. See also~\cite{To},~\cite{GS} and~\cite{Be}. This topic will not be addressed in this paper.

\subsection{Future work}
There are natural questions which arise from this work.
\begin{enumerate}
\item
What reasonable conditions on $\M$ and $\F$ assure that $\int_\M\F$ is cofibrantly generated, proper, etc., or is such relatively to $\M$?

\item
Assume that $\M$ is a simplicial model category. Is there a good notion of a simplicial functor from $\M$ to the 2-category of simplicial model categories which yields a good theory of a simplicial Grothendieck construction?
\item
How does the theory of model fibrations interact with other aspects of model category theory. For example, can it be used to compute homotopy limits of model categories in certain situations?

\end{enumerate}

\subsection{Organization }
This paper is organized as follows. We start in \S\ref{s:prelim} with various categorical preliminaries. In \S\ref{ss:Groth} we recall the Grothendieck construction and establish a toolkit of terminology. In \S\ref{ss:adj} and \S\ref{ss:left-right} we setup the basic relations between the Grothendieck construction and biCartesian fibrations in the $2$-categorical framework of categories and adjunctions. In \S\ref{ss:bicom} we give sufficient conditions for the Grothendieck construction to be (relatively) bicomplete (see Proposition~\ref{p:cocomplete}). Finally, in \S\ref{ss:model} we fix notations and terminology for the $2$-category of model categories.

In \S\ref{s:integral} we formulate the notion of a \textbf{proper relative} functor $\F: \M \lrar \ModCat$ and prove the existence of the integral model structure on $\int_\M\F$ for such functors (Theorem~\ref{model structure}). Using a result of Hinich we will show in \ref{ss:infinity-cat pov} that our construction gives a model-presentation of Lurie's $\infty$-categorical Grothendieck construction. In \S\ref{s:invariance} we will verify that this model structure is functorial and invariant under natural equivalences and suitable base changes (see Theorem~\ref{qa}). We will also prove that the integral model structure is well-behaved under iteration (see~\ref{p:fubini}).

In \S\ref{s:model-fib} we observe that the notion of a proper relative functor makes sense for a general category $\M$, as long as it is equipped with three distinguished classed of morphisms $\W_\M,\Cof_\M,\Fib_\M$. In this case, the Grothendieck construction cannot be, in general, a model category. However, the model category axioms will hold for $\int_\F \M$ in a \textbf{relatively} to $\M$. This leads naturally to the notions of \textbf{relative model category} and \textbf{model fibration} (see Definitions~\ref{d:relative-model} and Definition~\ref{d:model-fib}). The main result of this section is that the relative integral model structure induced an equivalence of $(2,1)$-categories between proper relative functors $\M \lrar \ModFib$ and model fibrations $\N \lrar \M$.

In \S\ref{s:examples} we will consider various classes of examples. We will begin in \S\ref{ss:slice} with the basic examples of slice and coslice categories. We then continue in \S\ref{ss:group-act} by organizing strict and weak group actions into suitable model fibrations and establishing a Quillen equivalence between them using Theorem~\ref{qa}. In \S\ref{ss:alg} we show that under the hypothesis of~\cite{SS} (see Definition~\ref{d:SS-assume}), a symmetric monoidal model category yields a proper relative functor given by associating to each associative algebra object its category of modules. In particular, this shows that most model structures for spectra (see Example~\ref{e:many}) are admissible with respect to the 2-coloured operad of algebras and modules. To the best of the authors' knowledge, this was only known for the positive model structure on symmetric spectra (\cite{EM}) and the positive model structure on orthogonal spectra (\cite{Kro}). Finally, in \S\ref{ss:calg} we will establish the validity of the analogous construction for \textbf{commutative algebras}, under suitable hypothesis.

\newpage 

\begin{ackname}
We would like to thank Michael Batanin, Ross Street, Geoffroy Horel and Tomer Schlank for useful discussions and suggestions.

The authors were supported by the Dutch Science Foundation (NWO), grants SPI 61-638 and 62001604 (respectively).  
\end{ackname}

\section{Preliminaries }\label{s:prelim}

Unless otherwise stated, a category will always mean a large category and a $2$-category will always mean a \textbf{very large weak 2-category}. If $\C$ is a category and $\D$ is a $2$-category, a functor $\C \lrar \D$ will always mean a \textbf{pseudo-functor}. We will denote by $\Fun(\C,\D)$ the (very large) 2-category of (pseudo-)functors, psuedo-natural transformations, and modifications.

\subsection{The Grothendieck construction }\label{ss:Groth}

We shall begin with the basic definitions. Let $p:\D\rar \C$ be a functor. A morphism $\phi:x\rar y$ in $\D$ is called \textbf{$p$-Cartesian} if for every object
$u\in \D$ and every pair of morphisms $\psi:u\rar y$ and $g:p(u)\rar p(x)$ such that $p(\psi)=p(\phi)g$ there exists a unique morphism $\gam:u\rar x$ such that $\phi \gam =\psi$.
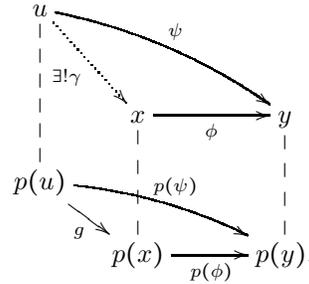
$$\xymatrix@=0.95pc{u\ar@{--}[ddd]\ar@{.>}[ddr]_{\exists !\gam}\ar@/^/[ddrrr]^\psi & & \\ &&& \\ & x\ar[rr]_{\phi}\ar@{--}[dd] & & y\ar@{--}[dd] \\p(u)\ar[dr]_g\ar@/^/[drrr]^{p(\psi)} & & \\ & p(x)\ar[rr]_{p(\phi)} & & p(y).}$$
Dually, a morphism $\phi:x\rar y$ in $\D$ is called \textbf{$p$-coCartesian} if for every object $v\in \D$ and every pair of morphisms $\psi:x\rar v$ and $g:p(y)\rar p(v)$ such that $gp(\phi)=p(\psi)$ there exists a unique morphism $\gam: y\rar v$
such that $\gam \phi=\psi$.
$$\xymatrix@=0.95pc{& & & v\ar@{--}[ddd]\\ & & \\ x\ar[rr]_{\phi}\ar@{--}[dd]\ar@/^/[uurrr]^{\psi} & & y\ar@{--}[dd]\ar@{.>}[uur]_{\exists ! \gam} \\ & & & p(v)\\ p(x)\ar@/^/[urrr]^{p(\psi)}\ar[rr]_{p(\phi)} &  & p(y)\ar[ur]_g.}$$

\begin{rem}\label{r:unique}
The uniqueness property of factorizations along $p$-coCartesian morphisms has the following direct implication. If $\phi: x \lrar y$ in $\D$ is $p$-coCartesian and $\gam,\gam': y \lrar z$ are morphisms such that $\gam \circ \phi = \gam' \circ \phi$ and $p(\gam) = p(\gam')$ then $\gam = \gam'$. The analogous statement for $p$-Cartesian morphisms holds as well.
\end{rem}

A functor $p:\D\rar \C$ is a \textbf{Cartesian fibration} if for every object $y\in \D$ and every morphism $f$ with target $p(y)$ there exist a $p$-Cartesian morphism $\phi$ with target $y$ such that $p(\phi)=f$.
$$\xymatrix@=0.95pc{\circ\ar@{--}[d]\ar[rr]^{\phi} & & y\ar@{--}[d]\\\circ\ar[rr]^{f=p(\phi)} && p(y).}$$
Dually, a functor $p:\D\rar \C$ is called a \textbf{coCartesian fibration} if for every object $x\in \D$ and every morphism $f$ with source $p(x)$ there exist a $p$-coCartesian morphism $\phi$ with source $x$ such that $f=p(\phi)$.
$$\xymatrix@=0.95pc{x\ar@{--}[d]\ar[rr]^{\phi} & & \circ\ar@{--}[d]\\p(x)\ar[rr]^{f=p(\phi)} && \circ.}$$
A functor $p:\D\rar \C$ is called a \textbf{biCartesian fibration} if it is both a Cartesian and a coCartesian fibration.

We will denote by $\GFib(\C)$ (resp. $\OpFib(\C)$) the $2$-category which has as objects the Cartesian (resp. coCartesian) fibrations over $\C$ and as morphisms the triangles
$$\xymatrix{\D\ar[dr]_{p} \ar[rr]^{\Phi} && \D'\ar[dl]^{p'}\\& \C &}$$
such that $\Phi$ sends $p$-Cartesian (resp. $p$-coCartesian) morphisms to $p'$-Cartesian (resp. $p'$-coCartesian) morphisms. The $2$-morphisms are natural transformations $\nu:\Phi\Rightarrow \Psi$ such that for every object $x\in \D$, $p'(\nu_x)=id_{p(x)}$.

\begin{rem}
The definition of a Cartesian fibration $\D \lrar \C$ is not invariant under replacing $\C$ with an equivalent category $\C' \simeq \C$. However, the category $\GFib(\C)$ \textbf{is} invariant under such equivalences. One possible way of making the definition itself invariant is to work with the equivalent notion of a \textbf{Street fibrations} (see~\cite{Str}). A similar solution may be applied to coCartesian fibrations.
\end{rem}

\begin{define}
Let $\F:\C\rar \Cat$ be a functor. The \textbf{Grothendieck construction} of $\F$ is the category $\int_\C\F$ defined as follows. An object of $\int_{\C} \F$ is a pair $(a,x)$ with $a\in \Obj\C$ and $x \in \Obj \F(a)$. A morphism $(a,x) \lrar (a',x')$ is a pair $(f,\phi)$ with $f:a\rar a'$ a morphism in $\C$ and $\phi:\F(f)(x)\rar x'$ a morphism in $F(a')$.
\end{define}

The Grothendieck construction carries a canonical functor
$$ p:\int_\C \F \lrar \C $$
given by the projection $(a,x) \mapsto a$, which is a coCartesian fibration. A $p$-coCartesian lift of $f: a \lrar b$ starting at $(a,x)$ if given by the morphism $(f,\Id):(a,x) \lrar (f(a),\F(f)(x))$. The association $\F \mapsto \int_\C\F$ determines an equivalence of $2$-categories
$$ \int :\Fun(\C,\Cat)\overset{\simeq}{\rar} \OpFib(\C) .$$
Dually, for a functor $\F:\C^{op}\rar \Cat$ we get a Cartesian fibration
$$ \pi:\int_{\C^{op}}\F\rar \C $$
yielding an equivalence of $2$-categories
$$ \int :\Fun(\C^{op},\Cat) \overset{\simeq}{\rar} \GFib(\C). $$

\subsection{Adjunctions and biCartesian fibrations}\label{ss:adj}

Recall that a $2$-category (see \cite{Mac}[XII \S 3]) is called a $(2,1)$-category if all $2$-morphisms are invertible. 
For any $2$-category $\C$ one can consider the associated $(2,1)$-category consisting of the same objects and $1$-morphisms but only the invertible $2$-morphisms of $\C$. Clearly, a $2$-functor $\C\lrar \D$ induces a functor between the associated $(2,1)$-categories. 
Following \cite{Mac}[IV \S 7-8] we will consider the $(2,1)$-category $\AdjCat$ of \textbf{adjunctions} defined as follows.
An object of $\AdjCat$ is a category $\C$, and a morphism from $\C$ to $\D$ is an adjunction
$\xymatrix{\C\ar[r]<1ex>^f & \D\ar[l]<1ex>_(0.55){\upvdash}^u}$. A $2$-morphism
$$ (\sig,\tau) : f \dashv u \Rightarrow f' \dashv u' $$
also called a \textbf{pseudo-transformation of adjunctions}, is a pair of natural \textbf{isomorphisms} $\sigma: f\Rightarrow f'$ and $\tau:u'\Rightarrow u$ such that for every $c\in \C,\;d\in \D$ the
following square of sets commutes
$$\xymatrix{\C(c,u'(d))\ar[r]^{\cong}\ar[d]_{\C(c,\tau_d)} &\D(f'(c),d)\ar[d]^{\D(\sigma_{c},d)}\\
\C(c,u(d))\ar[r]^{\cong} & \D(f(c),d)}$$
Vertical (resp. horizontal) composition is defined via term-wise vertical (resp. horizontal) composition of natural transformations. 
Given a category $\I$, we may consider the category $\Fun(\I,\AdjCat)$ which has as objects the functors $\I \lrar \AdjCat$ and as morphisms the pseudo-natural transformations. Explicitly, given $\F,\G:\I\rar \AdjCat$ a \textbf{pseudo-natural transformation} $\Phi:\F\Rightarrow \G$ assigns for each $i\in \I$ an adjunction $\Phi(i):\xymatrix{\F(i)\ar[r]<1ex>^g & \G(i)\ar[l]<1ex>_(0.55){\upvdash}^v}$ and for each arrow $\alpha:i\rar i'\in \I$, a pseudo-transformation of adjunctions $\Phi(i)=(\sigma:g'f\overset{\cong}{\Rightarrow} f'g ;\;\tau:vu'\overset{\cong}{\Rightarrow} uv')$ rendering the following square commutative up to isomorphism
$$\xymatrix{\F(i)\ar[r]<1ex>^f\ar[d]<-1ex>_g & \F(i')\ar[l]<1ex>_(0.5){\upvdash}^{\overset{u}{\LSwarrow}}\ar[d]<1ex>^{g'} \\ \G(i)\ar[r]<-1ex>_{f'}\ar[u]<-1ex>^(0.5){\leftvdash}_{v} & \G(i')\ar[l]<-1ex>^(0.5){\downvdash}_{u'}\ar[u]<1ex>_(0.5){\vdash}^{v'};}$$ and this data is subject to coherence conditions (see \cite{Gra}). Composition of pseudo-natural transformations is defined in the obvious way.

Let $U_L: \AdjCat \lrar \Cat$ be the $(2,1)$-functor which is identity on objects and associates to each $1$-morphism
$$
\xymatrix{
\C\ar[r]<1ex>^f & \D\ar[l]<1ex>_(0.55){\upvdash}^u
}$$
The left functor $f: \C \lrar \D$, and to each $2$-morphism $ (\sig,\tau) : f \dashv u \Rightarrow f' \dashv u' $ the pseudo-natural transformation $\sig: f \Rightarrow f'$. Similarly, let us denote by $U_R: \AdjCat \lrar \Cat^{\op}$ the $(2,1)$-functor which associates to each adjunction its right functor and to each $2$-morphism $(\sig,\tau)$ the right part $\tau$. The $(2,1)$-functors $U_L$ and $U_R$ are \textbf{faithful}, in the sense that for each pair of categories $\C,\D$ the induced functors
$$ (U_L)_*: \AdjCat(\C,\D) \lrar \Cat(\C,\D) $$
and
$$ (U_R)_*: \AdjCat(\C,\D) \lrar \Cat(\D,\C) $$
are \textbf{fully-faithful}. Similarly, for every small category $\I$, we have induced faithful $(2,1)$-functors
$$ (U_L)^{\I}: \AdjCat^{\I} \lrar \Cat^{\I} $$
and
$$ (U_R)^{\I}: \AdjCat^{\I} \lrar (\Cat^{\op})^{\I} = \Cat^{\I^{\op}} $$
We will denote by $\BiFib(\I)$ the $2$-category of biCartesian fibration over $\I$. The morphisms in $\BiFib(\I)$ are given by adjunctions
$$\xymatrix{
\J\ar[dr]_{p} \ar[rr]<1ex>^{\Phi} && \J'\ar[dl]^{p'}\ar^{\Psi}[ll]<1ex>_{\upvdash}\\& \I &}$$
such that $\Phi$ preserves coCartesian morphisms and $\Psi$ preserves Cartesian morphisms. The $2$-morphisms are given by psuedo-transformations of adjunctions. One then has similar left/right forgetful functors
$$ V_L: \BiFib(\I) \lrar \OpFib(\I) $$
and
$$ V_R: \BiFib(\I) \lrar \GFib(\I) $$
which are again faithful (this is because the right adjoint of any functor that preserves coCartesian morphisms preserves Cartesian morphisms and vice versa). One then has a natural commutative diagram
$$ \xymatrix{
& \Cat^{\I} \ar_(0.4){\simeq}^{\int_{\I}}[r]& \OpFib(\I) \\
\AdjCat^{\I} \ar^{\int_{\I}}[r]\ar^{U^{\I}_L}[ur]\ar_{U^{\I}_R}[dr] & \BiFib(\I) \ar_{V_L}[ur]\ar^{V_R}[dr] & \\
& \Cat^{\I^{\op}} \ar^{\simeq}_{\int_{\I^{\op}}}[r] & \GFib(\I) \\
}$$

The following proposition seems to be well known to experts, but we were not able to find a proof in the literature.
\begin{pro}\label{p:bifib}
The $(2,1)$-functor
$$ \int :\Fun(\I,\AdjCat)\overset{\simeq}{\lrar} \BiFib(\I) .$$
is an equivalence of $(2,1)$-categories.
\end{pro}
\begin{proof}
Since $U^I_L$ and $V_L$ are faithful this amounts to verifying the following claims:
\begin{enumerate}
\item
Let $\F: \I \lrar \Cat$ be a functor such that the coCartesian fibration $\int_\I\F \lrar \I$ is in the image of $V_L$. Then $\F$ is in the image of $U^{\I}_L$.
\item
Let $\F,\G: \I \lrar \AdjCat$ be functors and $\sig: U^{\I}_L(\F) \Rightarrow U^{\I}_L(\G)$ be a pseudo-natural transformation such that the induced map $\sig_*: \int_\I\F \lrar \int_\I\G$ is in the image of $V_L$. Then there exists a morphism $\ovl{\sig}: \F \Rightarrow \G$ such that $U^{\I}_L(\ovl{\sig})) = \sig$.
\end{enumerate}
Assertion $(1)$ can be found, for example, in~\cite{}. As for assertion $(2)$, the pseudo-natural transformation $\sig$ determines a functor $\Sig: \I \times [1] \lrar \Cat$ whose restrictions to $\I \times \{0\}$ and $\I \times \{1\}$ are $U^{\I}_L(\F)$ and $U^{\I}_L(\G)$, respectively. On the other hand, let
$$\xymatrix{
\int_\I\F\ar[dr] \ar[rr]<1ex>^{\sig_*} &&  \int_\I\G\ar[dl]\ar^{\tau_*}[ll]<1ex>_{\upvdash}\\& \I &}$$
be the morphism in $\BiFib(\I)$ whose image under $V_L$ is $\sig_*$. Then the adjunction $\sig_* \dashv \tau_*$ determines a biCartesian fibration
$$ \M \lrar \I \times [1] $$
whose pullbacks to $\I \times \{0\}$ and $\I \times \{1\}$ are $\int_\I\F$ and $\int_\I\G$ respectively. According to $(1)$, the object $\M$ is the Grothendieck construction of a pseudo-natural transformation $\ovl{\Sig}:\I \times [1] \lrar \AdjCat$ whose image under $U^{\I \times [1]}_L$ is $\Sig$. The functor $\ovl{\Sig}$, in turn, determines as morphism $\ovl{\sig}:\F \Rightarrow \G$ in $\AdjCat^{\I}$ such that $V_L(\ovl{\sig}) = \sig$.
\end{proof}

\subsection{Base change for diagrams of adjunctions }\label{ss:left-right}
Let
$$ \xymatrix{
\I \ar[dr]_\F\ar[rr]<1.3ex>^{\L} & & \J\ar[dl]^\G\ar[ll]<0.7ex>^{\R}_\upvdash \\ &  \AdjCat & \\}$$
be a (not necessarily commutative) diagram of categories such that the horizontal pair forms an adjunction. 
\begin{define}\label{d:left-right}
A \textbf{left morphism} from $\F$ to $\G$ over $\L \dashv \R$ is a pseudo-natural transformation $\F \Rightarrow \G \circ \L$, i.e., a compatible family of adjunctions.
$$\xymatrix{
\Sig^L_A:\F(A) \ar[r]<1ex> & \G(\L(A)):\Sig^R_A\ar[l]<1ex>_(0.55){\upvdash}.}$$
indexed by $A \in \I$. Similarly, a \textbf{right morphism} from $\F$ to $\G$ is a pseudo-natural transformation $\F \circ \R \Rightarrow \G$, i.e., a compatible family of adjunctions
$$\xymatrix{
\Theta^L_B:\F(\R(B)) \ar[r]<1ex> & \G(B):\Theta^R_B\ar[l]<1ex>_(0.45){\upvdash}.}$$
indexed by $B \in \J$.
\end{define}

\begin{rem}
Throughout this subsection we will be dealing with a pair of functor $\F,\G$ into $\AdjCat$ with different domains. To keep the notation simple, we shall, as before, use the notation $f_{!} \dashv f^*$ to indicate the image of a morphism $f$ under either $\F$ or $\G$. The possible ambiguity can always be resolved since $\F$ and $\G$ have different domains.
\end{rem}

Now let $\left(\Sig^L, \Sig^R\right)$ be a left morphism as above. We define an adjunction
$$
\xymatrix{\Phi^L:\int_{\I}\F\ar[r]<1ex> & \int_{\J}\G:\Phi^R\ar[l]<1ex>_(0.5){\upvdash}.}
$$
as follows. Given $(A,X) \in \int_\I\F$ we define
$$ \Phi^L(A,X) = \left(\L(A),\Sig^L_A(X)\right) $$
and
$$ \Phi^R(B,Y) = \left(\R(B),\Sig^R_{\R (B)}(\eps^*Y)\right) $$
where $\eps: \L\R(B) \lrar B$ is the counit map. The action on morphisms is defined in the obvious way using the structure of $\left(\Sig^L, \Sig^R\right)$ as a pseudo-natural transformation. The counit 
$$ \Phi^L\left(\Phi^R(B,Y)\right) = \left(\L\R(B),\Sig^L_{\R(B)}\Sig^R_{\R(B)}(\eps^*Y)\right) \lrar (B,Y) $$
is given by the pair of the counit maps $\eps: \L\R(B) \lrar B$ and 
$$ \Sig^L_{\R(B)}\Sig^R_{\R(B)}(\eps^*Y) \lrar \eps^*Y $$
It is routine to verify that this counit map exhibits an adjunction $\Phi^L \dashv \Phi^R$. 

\begin{rem}\label{r:explicit}
More explicitly, if 
$$ (f,\vphi): (A,X) \lrar \Phi^R(B,Y) = \left(\R(B),\Sig^R_{\R(B)}(\eps_{!}Y)\right) $$
is a morphism in $\int_\M \F$, where $f: A \lrar \R(B)$ is a morphism in $\M$ and 
$ \vphi: f_{!}X \lrar \Sig^R_B(\eps_{!}Y) $
is a morphism in $\F(\R(B))$, then its adjoint morphism
$$ (f,\vphi)^{\ad}: \Phi^L(X,A) = (\L(X),\Sig^L_A(X)) \lrar (B,Y) $$
is given by the pair $(f^{\ad},\psi)$ where
$ f^{\ad}: \L(A) \lrar B $
is the adjoint of $f$ with respect to $\L \dashv \R$ and
$$ \psi: (f^{\ad})_{!}\Sig^L_A(X) \lrar Y $$ 
is the adjoint of
$$ \vphi^{\ad}: X \lrar f^*\Sig^R_{\R(B)}(\eps^*Y) \cong \Sig^R_A (\L(f))^*(\eps^*Y) \cong \Sig^R_A \left(f^{\ad}\right)^*Y $$
with respect to the adjunction $\left(f^{\ad}\right)_{!} \circ \Sig^L_A   \dashv  \Sig^R_A \circ \left(f^{\ad}\right)^*$.
\end{rem}

\begin{rem}
If the base change $\L \dashv \R$ is the identity adjunction then the notions of left and right morphisms coincide and become the notion of a morphism in $\AdjCat^{\I}$. In this case the associated adjunction $\Phi^L \dashv \Phi^R$ is just the associated morphism in $\BiFib(\I)$.
\end{rem}

\begin{rem}\label{r:dual}
The dual case of a \textbf{right morphism} 
$$\xymatrix{
\Theta^L_B:\F(\R(B))\ar[r]<1ex> & \G(B) :\Theta^R_B\ar[l]<1ex>_(0.45){\upvdash}}$$
works in a similar way and one obtains an induced adjunction
$$
\xymatrix{\Psi^L:\int_{\I}\F\ar[r]<1ex> & \int_{\J}\G:\Psi^R\ar[l]<1ex>_(0.5){\upvdash}}
$$
\end{rem}

\subsection{Relative limits and colimits of biCartesian fibrations}\label{ss:bicom}
Let $\I$ be a small category. We will denote by $\I^{\triangleright}$ the category obtained by formally adding to $\I$ a \textbf{terminal object}. More explicitly, $\Ob(\I^{\triangleright}) \x{\df}{=} \Ob(\I) \cup \{\ast\}$, the inclusion $\I \hrar \I^{\triangleright}$ is fully-faithful, and $\Hom_{\I^{\triangleright}}(X,\ast)$ is a one pointed set for every $X \in \Ob(\I^{\triangleright})$. Similarly, we will denote by $\I^{\triangleleft}$ the category obtained by formally adding to $\I$ an \textbf{initial object}.
\begin{define}
Let $\I$ be a small category and consider a commutative diagram of categories of the form
\begin{equation}\label{e:square4}
\xymatrix{
\I \ar^{\del}[r]\ar[d] & \D \ar^{\pi}[d] \\
\I^{\triangleright} \ar_{\eps}[r] & \C \\
}
\end{equation}
A \textbf{colimit of $\del$ relative to $\eps$ and $\pi$} is a dashed lift
$$ \xymatrix{
\I \ar^{\del}[r]\ar[d] & \D \ar^{\pi}[d] \\
\I^{\triangleright} \ar_{\eps}[r]\ar@{-->}[ur] & \C \\
}$$
which is initial in the category of all such lifts. We will say that $\pi: \D \lrar \C$ is \textbf{cocomplete}, or that $\D$ is $\pi$-cocomplete, if for every square such as~\ref{e:square4}~, $\del$ admits a colimit relative to $\eps$ and $\pi$. Dually, one may define the notion of \textbf{relative limits} (and relative completeness) by using $\I^{\triangleleft}$ instead of $\I^{\triangleright}$ and taking the terminal dashed lift instead of the initial. We will say that a functor is bicomplete if it is both complete and cocomplete.
\end{define}

\begin{rem}
A category $\D$ is (co)complete if and only if the terminal map $\D \lrar *$ is (co)complete.
\end{rem}

\begin{rem}\label{r:cocompose}
It is straightforward to verify that the class of cocomplete functors is closed under composition. Hence if $\pi: \D \lrar \C$ is a cocomplete functor and $\C$ is a cocomplete category then $\D$ is cocomplete as well. Furthermore, in this case $\pi$ will preserves all colimits. The analogous statement for complete functors holds as well.
\end{rem}

The following Proposition appears to be known to experts. Since we were not able to find an explicit proof in the literature, we have included, for the convenience of the reader, the details of the argument.
\begin{pro}\label{p:cocomplete}
Let $\C$ be a category and Let $\F: \C \lrar \Cat$ be a functor. Assume that $\F(A)$ has all small colimits for every $A$. Then $\int_{\C}\F \lrar \C$ is cocomplete.
\end{pro}
\begin{proof}
Let
$$
\xymatrix{
\I \ar^{\del}[r]\ar[d] & \int_\C\F \ar^{\pi}[d] \\
\I^{\triangleright} \ar_{\eps}[r] & \C \\
}
$$
be a square. For convenience, let us write
$$ \del(i) = (\del_0(i),\del_1(i)) $$
where $\del_0 = \pi \circ \del$ and $\del_1(i) \in \F(\del_0(i))$ is the object determined by $\del(i)$. Furthermore, for each $\alp: i \lrar j$ in $\I$ we will denote by $\del_1(\alp): \del_0(\alp)_*\del_1(i) \lrar \del_1(j)$ the morphism specified by $\del$.

Let us denote by $* \in \I^{\triangleright}$ the cone point and by $\theta_i: i \lrar *$ the unique map in $\I^{\triangleright}$ from $i \in \I \subseteq \I^{\triangleright}$ to $*$.
Let
$$ \del_1': \I \lrar \F(\eps(*)) $$
be given by $\del_1'(i) = \eps(\theta_i)_*(\del_1(i))$ for every object $i \in \I$ and $\del_1'(\alp) = \eps(\theta_i)_*(\del_1(\alp))$ for every morphism $\alp: j \lrar i$.  We then define a functor
$$ \del': \I \lrar \int_{\C}\F $$
by setting
$$ \del'(i) = \left(\eps(*),\del_1'(i)\right) $$
Now since $\F\left(\eps(*)\right)$ has small colimits we can extend $\del_1'$ to a colimit diagram
$$ \ovl{\del_1'}: \I^{\triangleright} \lrar \F\left(\eps(*)\right) $$
The extension $\ovl{\del_1'}$ determines a dashed lift
$$ \xymatrix{
\I \ar^{\del}[r]\ar[d] & \int_\C\F \ar^{\pi}[d] \\
\I^{\triangleright} \ar_{\eps}[r]\ar^{\ovl{\del}}@{-->}[ur] & \C \\
}$$
given by $\ovl{\del}(i) = \del(i)$ for $i \in \I$ and $\ovl{\del}(*) = \left(\eps(*),\ovl{\del_1'}(*)\right)$.

It is left to show that $\ovl{\del}$ is initial in the category of all such lifts. Let $\eta: \I^{\triangleright} \lrar \int_\M\F$ be a competing lift. By (uniquely) factoring the morphisms
$$ \eta(\theta_i):\del(i)=\eta(i) \lrar \eta(*) $$
as
$$ \xymatrix{
\del(i) \ar[rr]\ar[dr] && \eta(*) \\
& (\eps(*),\del'_1(i)) \ar[ur] & \\
}$$
along the $\pi$-coCartesian morphisms $\del(i) \lrar (\eps(*),\del'_1(i))$, we obtain a natural transformation from $\del'_1$ to the constant map $\I \lrar \F(\eps(*))$ on $\eta(*)$. This induces a natural map $\ovl{\del_1'}(*) \lrar \eta(*)$ in $\F(\eps(*))$ and hence a natural transformation of lifts
$$ \tau:\ovl{\del} \lrar \eta .$$
It is then straightforward to verify the uniqueness of $\tau$.

\end{proof}

%

\subsection{Model categories }\label{ss:model}
We shall use a strengthened version of Quillen's original definition of a (closed) model category \cite{Qui}.

\begin{defn}\label{d:model-cat}
A \textbf{model category} $\M$ is category with three distinguished classes of morphisms
$\W=\W_{\M}\;$, $\C of=\C of_{\M}\;$,  $\F ib=\F ib_{\M}$ called \textbf{weak equivalences}, \textbf{fibrations} and \textbf{cofibrations} (respectively), satisfying the following axioms:
\begin{enumerate}[MC1]
\item (Bicompleteness) The category $\M$ is complete and cocomplete.
\item (Two-out-of-three) If $f,g$ are composable maps such that two of $f,g$ and $gf$ are in $\W$ then so is the third.

\item (Retracts) The classes $\W,\Fib$ and $\Cof$ contain all isomorphisms and are closed under retracts.
\item (Liftings) Given the commutative solid diagram in $\M$
$$\xymatrix{A\ar[d]_{i}\ar[r] & X\ar[d]^p\\ B\ar[r]\ar@{-->}[ur] & Y\\}$$ in which $i\in \Cof$ and $p\in \Fib$, a dashed arrow  exists if either $i$ or $p$ are in $\W$.

\item (Factorizations) Any map $f$ in $\M$ has two functorial factorizations:
\begin{enumerate}[(i)]
\item $f=pi$ with $i\in \Cof$ and $p\in \Fib\cap \W$;
\item $f=qj$ with $j\in \Cof \cap \W$ and $q\in \Fib$.
\end{enumerate}
\end{enumerate}

The maps in $\Fib\cap \W$ (resp. $\Cof\cap \W$) are referred to as \textbf{trivial fibrations} (resp. \textbf{trivial cofibrations}). For an object $X\in \M$ we denote by $X^{\fib}$ (resp. $X^{\cof}$) the functorial fibrant (resp. cofibrant) replacement of $X$, obtained by factorizing the map to the terminal object $X\lrar *$ (resp. from the initial object $\emptyset \lrar X$) into a trivial cofibration followed by a fibration (resp. a cofibration followed by trivial fibration).
\end{defn}

Morphisms of model categories are defined as follows.

\begin{defn}
Let $\M,\N$ be model categories. An adjunction $$\xymatrix{F:\M\ar[r]<1ex> & \N:U\ar[l]<1ex>_{\upvdash}}$$ is called a  \textbf{Quillen adjunction} if $F$ preserves cofibrations and trivial cofibration, or equivalently, if $U$ preserves fibrations and trivial fibrations. In such a case, $F$ is called a \textbf{left Quillen functor} and $U$ is called a \textbf{right Quillen functor}.A Quillen adjunction $F\dashv G$ is called a \textbf{Quillen equivalence} if for every cofibrant object $X\in \M$, the composite $X \lrar U(F(X)) \lrar U(F(X)^{\fib})$ is a weak equivalence in $\M$, and for every fibrant object $Y\in \N$, the composite $F(U(Y)^{\cof}) \lrar F(U(Y)) \lrar Y$ is a weak equivalence in $\N$.
\end{defn}

\begin{defn}
We will denote by $\ModCat$ the $(2,1)$-category which has as objects the model categories and as morphisms the Quillen adjunctions; the source and target being those of the left Quillen functor. The $2$-isomorphisms are given by the \textbf{pseudo-natural transformations of Quillen adjunctions}, i.e. the pseudo-natural transformations of the underlying adjunctions in the sense of \S\ref{ss:adj}.
\end{defn}


\section{The integral model structure}\label{s:integral}

Suppose $\M$ is a model category and $\F:\M\rar \ModCat$ a functor. For a morphism $f:A\lrar B$ in $\M$, we denote the associated adjunction in $\ModCat$ by
$$\xymatrix{f_{!}:\F(A)\ar[r]<1ex> & \F(B):f^*\ar[l]<1ex>_{\upvdash}.}$$
Recall that an object of the Grothendieck construction $\int_{\M} \F$ is a pair $(A,X)$ where $A\in \Obj\M$ and $X\in \Obj\F(A)$ and a morphism $(A,X)\rar (B,Y)$ in $\int_{\M} \F$ is a pair $(f,\phi)$ where $f:A\rar B$ is a morphism in $\M$ and $\phi:f_{!}X\rar Y$ is a morphism in $\F(B)$. In this case, we denote by $\phi^{ad}:X\rar f^*Y$ the adjoint map of $\phi$.
\begin{defn}\label{d:model}
Call a morphism $(f,\phi):(A,X)\rar (B,Y)$ in $\int_{\M} \F$
\begin{enumerate}
\item a \textbf{weak equivalence} if $f:A\rar B$ is a weak equivalence in $\M$ and the composite $f_{!}(X^{cof})\rar f_{!}X\rar Y$ is a weak equivalence in $\F(B)$;
\item a \textbf{fibration} if $f:A\rar B$ is a fibration and $\phi^{ad}:X\rar f^*Y$ is a fibration in $\F(A)$;
\item a \textbf{cofibration} if $f:A\rar B$ is a cofibration in $\M$ and $\phi:f_{!}X\rar Y$ is a cofibration in $\F(B)$.
\end{enumerate}
We denote these classes by $\W$, $\Fib$ and $\Cof$ respectively.
\end{defn}

\begin{rem}
The reason for defining weak equivalences in $\int_\M\F$ via a cofibrant replacement of the domain is due to the fact that $f_!(X)$ itself might not have the correct homotopy type if $X$ is not cofibrant.
\end{rem}

We now turn to set up the appropriate conditions on a functor $\F:\M\lrar \ModCat$ that will guarantee the existence of a model structure with weak equivalences, fibrations and cofibrations as in Definition \ref{d:model}. An a-priori condition one should impose is the following.
\begin{defn}\label{d:relative}
We will say that a functor $\F: \M \lrar \ModCat$ is \textbf{relative} if for every weak equivalence $f:A\lrar B$ in $\M$, the associated Quillen pair $f_{!}\dashv f^*$ is a Quillen equivalence.
\end{defn}

\begin{rem}
The condition of relativeness is essential if one wishes to consider $\F$ from an $\infty$-categorical point of view. We will address this issue in greater detail in \S\ref{ss:infinity-cat pov}. 
\end{rem}

\begin{obs}\label{o:sym}
For a relative functor $\F:\M\lrar \ModCat$, a map $(f,\phi):(A,X)\lrar (B,Y)$ in $\int_{\M} \F$ is a weak equivalence iff $f:A\lrar B$ is a weak equivalence and the composite $X\lrar f^*Y\lrar f^*(Y^{\fib})$ is a weak equivalence.
In other words, for a relative functor $\F$, the definition of weak equivalences in $\int_\M \F$ is symmetric with respect to the adjuctions induced by $\F$. This may fail for non-relative functors, see Example~\ref{e:non-relative} below.
\end{obs}

\begin{defn}\label{d:proper}
Let $\M$ be a model category and $\F:\M\rar \ModCat$ a functor. We shall say that $\F$ is
\begin{enumerate}
\item
\textbf{left proper} if whenever  $f:A \lrar B$ is a trivial cofibration in $\M$ the associated left Quillen functor preserves weak equivalences, i.e., $f_{!}(\W_{\F(A)})\subseteq \W_{\F(B)}$;
\item
\textbf{right proper} if whenever $f:A \lrar B$ is a trivial fibration in $\M$ the associated right Quillen functor preserves weak equivalences, i.e., $f^*(\W_{\F(A)})\subseteq \W_{\F(B)}$.
\end{enumerate}
We shall say that $\F$ is \textbf{proper} if it is both left and right proper.
\end{defn}

\begin{rem}
We will see later that the condition of properness is essential for the construction of a well-behaved model structure on $\int_\M\F$ (see Corollary~\ref{c:necessary} below).
\end{rem}

\begin{lem}\label{characterization}
Let $\F:\M \lrar \ModCat$ be a proper relative functor.
\begin{enumerate}[(i)]
\item A morphism $(f,\phi):(A,X)\lrar (B,Y)$ is in $\Cof \cap \W$ if and only if $f:A\lrar B$ is a trivial cofibration and $\phi:f_{!}X\lrar Y$ is a trivial cofibration in $\F(B)$;
\item A morphism $(f,\phi):(A,X)\lrar (B,Y)$ is in $\Fib \cap \W$ if and only if $f:A\lrar B$ is a trivial fibration and $\phi^{ad}:X\lrar f^*Y$ is a trivial fibration in $\F(A)$.
\end{enumerate}
\end{lem}
\begin{proof}\
\begin{enumerate}[(i)]
\item
Since the map $X^{\cof}\overset{\sim}{\lrar} X$ is a weak equivalence, and $f_{!}$ preserve weak equivalences, the composite
$$f_{!}(X^{cof})\lrar f_{!}X\lrar Y$$
is a weak equivalence if and only if $f_{!}X\lrar Y$ is a weak equivalence.
\item
In light of Observation~\ref{o:sym} we can just use the dual argument. More explicitly, since the map $X \x{\sim}{\lrar} X^{\fib}$ is a weak equivalence, and $f^*$ preserve weak equivalences, the composite
$$ X \lrar f^*Y \lrar f^*(Y^{\fib}) $$
is a weak equivalence if and only if $X \lrar f^*Y$ is a weak equivalence.
\end{enumerate}
\end{proof}

We are now ready to prove the main theorem of this section.
\begin{thm}\label{model structure}
Let $\M$ be a model category and $\F: \M \lrar \ModCat$ a proper relative functor. The classes of weak equivalences $\W$, fibrations  $\Fib$ and cofibrations $\Cof$ of~\ref{d:model} endow $\int_{\M}\F$ with the structure of a model category, called the \textbf{integral model structure}.
\end{thm}
The proof of Theorem~\ref{model structure} will occupy the reminder of this section and we shall break it down into separate claims, each verifying a different axiom. We use the terminology of Definition~\ref{d:model} for convenience.

\begin{enumerate}
\item
\textbf{$\int_{\M}\F$ is bicomplete}.
\begin{proof}
Viewing $\F$ as a diagram of left (Quillen) functors indexed by $\M$, Proposition~\ref{p:cocomplete} shows that $\int_{\M}\F\lrar \M$ is cocomplete and since $\M$ is cocomplete it follows that $\int_{\M}\F$ is cocomplete. Viewing $\F$ as a diagram of right (Quillen) functors indexed by $\M^{\op}$, Proposition~\ref{p:cocomplete} applies again to show that $\left(\int_{\M}\F\right)^{\op}\lrar \M^{\op}$ is cocomplete. Since $\M^{\op}$ is cocomplete it follows that $\left(\int_{\M}\F\right)^{\op}$ is cocomplete so that $\int_{\M}\F$ is complete.   


\end{proof}

\item
\textbf{$\W$ is closed under 2-out-of-3}.
\begin{proof}
Let $(A,X) \x{(f,\phi)}{\lrar} (B,Y) \x{(g,\psi)}{\lrar} (C,Z)$ be a pair of composable morphisms, where $\phi: f_{!}X \lrar Y$ is a morphism in $\F(B)$ and $\psi: g_{!}Y \lrar Z$ is a morphism in $\F(Z)$. Their composition is given by the map
$$ (g \circ f,\psi \circ g_{!}(\phi)) : (A,X) \lrar (C,Z) $$
Assume first that $(f,\phi)$ is a weak equivalences. By definition we get that $f$ is a weak equivalence in $\M$. Let $X^{\cof} \lrar X$ be a cofibrant replacement of $X$. Since $(f,\phi)$ is a weak equivalence the composition
$$ f_{!}\left(X^{\cof}\right) \lrar f_{!}(X) \x{\phi}{\lrar} Y $$
is a weak equivalence. Since $X^{\cof}$ is cofibrant and $f_{!}$ is left Quillen we get that $f_{!}\left(X^{\cof}\right)$ is cofibrant and so by the above we can consider it as a cofibrant replacement of $Y$. Now suppose that one (and hence both) of $g,f \circ g$ are weak equivalences in $\M$. Then we see that the condition of $(g,\psi)$ being a weak equivalence and the condition of $(g \circ f,\psi \circ g_{!}(\phi))$ being a weak equivalence are both equivalent to the composition
$$ g_{!}f_{!}\left(X^{\cof}\right) \lrar g_{!}(Y) \x{\psi}{\lrar} Z $$
being a weak equivalence in $\F(C)$.

Now assume that $(g,\phi)$ and $(g \circ f,\psi \circ g_{!}(\phi))$ are weak equivalences. Then $g$ and $g \circ f$ are weak equivalences and so $f$ is a weak equivalence. Let $X^{\cof} \lrar X$ be a cofibrant replacement for $X$. We need to show that the composition
$$ f_{!}\left(X^{\cof}\right) \lrar f_{!}(X) \x{\phi}{\lrar} Y $$
is a weak equivalence in $\F(B)$. Factor this map as $f_{!}\left(X^{\cof}\right) \lrar Y^{\cof} \lrar Y$ where the former is a cofibration and the latter a weak equivalence. Then $Y^{\cof}$ can be considered as a cofibrant replacement for $Y$. Now consider the two maps
$$ g_{!}f_{!}\left(X^{\cof}\right) \lrar g_{!}\left(Y^{\cof}\right) \lrar Z $$

Since $(g,\psi)$ and $(g \circ f,\psi \circ g_{!}(\phi))$ are weak equivalences we get the right map and the composition of the two maps are weak equivalences. Hence we get that the map
$$ g_{!}f_{!}\left(X^{\cof}\right) \lrar g_{!}\left(Y^{\cof}\right) $$
is a weak equivalence in $\F(C)$. Since $g$ is a weak equivalence in $\M$ it follows that $g_{!} \dashv g^*$ is a Quillen equivalence. Since Quillen equivalences reflect equivalences between cofibrant objects we get that the map
$$ f_{!}\left(X^{\cof}\right) \lrar Y^{\cof} $$
is a weak equivalence, and hence the composition
$$ f_{!}\left(X^{\cof}\right) \lrar Y $$
is a weak equivalence.
\end{proof}
\item
\textbf{The class $\Cof$ (resp. $\Cof \cap \W$) satisfies the left lifting property with respect to $\W \cap \Fib$ (resp. $\Fib$)}.
\begin{proof}
Consider a commutative diagram of the form
\begin{equation}\label{e:square}
\xymatrix{
(A,X) \ar[r]\ar_{(i,\iota)}[d] & (B,Y) \ar^{(p,\pi)}[d] \\
(C,Z) \ar[r] & (D,W) \\
}
\end{equation}
where $(i,\iota)$ is a cofibration and $(p,\pi)$ is a fibration. We need to show that if either $(i,\iota)$ or $(p,\pi)$ is trivial then this diagram admits a lift. Now note that $p$ is a fibration in $\M$ and $i$ is a cofibration in $\M$ and in this case one of these will be a weak equivalence in $\M$. Hence the indicated lift
$$ \xymatrix{
A \ar[r]\ar_{i}[d] & B \ar^{p}[d] \\
C \ar[r]\ar^{f}@{-->}[ur] & D \\
}
$$
will exist. In order to extend $f$ to a lift in~\ref{e:square} we need to construct a map $\vphi: f_{!}C \lrar B$ satisfying suitable compatibility conditions. Unwinding the definitions this amounts to constructing a lift in the square
\begin{equation}\label{e:square2}
\xymatrix{
f_{!}i_{!}X \ar[r]\ar_{f_{!}\iota}[d] & Y \ar^{\pi^{\ad}}[d] \\
f_{!}Z \ar[r]\ar@{-->}^{\vphi}[ur] & p^*W \\
}\end{equation}
which lives in the model category $\F(B)$. Since $f_{!}$ is a left Quillen functor and $\iota$ is a cofibration we get that $f_{!}\iota$ is a cofibration. Similarly, by our assumptions  $\pi^{\ad}$ is a fibration. Now According to Lemma~\ref{characterization}, if $(i,\iota)$ is trivial then $\iota$ is trivial and if $(p,\pi)$ is trivial then $\pi^{\ad}$ is trivial. Since $f_{!}$ preserves trivial cofibrations we get that in either of these cases the indicated lift in~\ref{e:square2} will exist.
\end{proof}

\item
\textbf{Every morphism $f$ in $\int_\C\F$ can be functorially factored as a morphism $f' \in \Cof$ (resp. $f' \in \W \cap \Cof$) followed by a morphism $f'' \in \W \cap \Fib$ (resp. $f'' \in \Fib$).}
\begin{proof}
Let
$$ (f,\vphi): (A,X) \lrar (B,Y) $$
be a morphism in $\int_\C \F$, so that $f: A \lrar B$ is a morphism in $\M$ and $\vphi: f_{!}X \lrar Y$ is a morphism in $\F(B)$. We start by (functorially) factoring $f$ as
$$ A \x{f'}{\lrar} C \x{f''}{\lrar} B $$
where $f'$ is a cofibration (resp. trivial cofibration) in $\M$ and $f''$ is a trivial fibration (resp. fibration) in $\M$. Now consider the map
$$ \psi: f'_{!}X \lrar (f'')^*Y $$
in $\F(C)$ which is adjoint to
$$ \vphi: f''_{!}f'_{!}X = f_{!}X \lrar Y $$
We can (functorially) factor $\psi$ as
$$ f'_{!}X \x{\vphi'}{\lrar} Z \x{\psi'}{\lrar} (f'')^*Y $$
where $\vphi'$ is a cofibration (resp. trivial cofibration) in $\F(C)$ and $\psi'$ is a trivial fibration (resp. fibration) in $\F(C)$. Let
$$ \vphi'': f''_{!}Z \lrar Y $$
be the adjoint map of $\psi'$. Then we obtain a factorization
$$ (A,X) \x{(f',\vphi')}{\lrar} (C,Z) \x{(f'',\vphi'')}{\lrar} (B,Y) $$
of $(f,\vphi)$ and using Lemma~\ref{characterization} we get that $(f',\vphi')$ is a cofibration (resp. trivial cofibration) in $\int_\M \F$ and $(f'',\vphi'')$ is a trivial fibration (resp. fibration) in $\int_\M \F$.

\end{proof}

\item
\textbf{$\W,\Cof$ and $\Fib$ are closed under retracts and contain all isomorphisms}.

\begin{proof}
We have already established that every map in $\int_\M\F$ can be \textbf{functorially} factored into a cofibration followed by a trivial fibration and to a trivial cofibration followed by a fibration. Furthermore, in light of the $2$-out-of-$3$ rule verified above every weak equivalence in $\int_\M\F$ will be factored into a trivial cofibration followed by a trivial fibration in any of these two factorizations. By applying these factorizations to retract diagrams we see that $\W$ will be closed under retracts once $\W \cap \Cof$ and $\W \cap \Fib$ are closed under retracts.

Given a category $\C$, we will use the term \textbf{retract diagram of maps in $\C$} to indicate a retract diagram in the arrow category $\C^{[1]}$. Now consider a retract diagram of maps
$$ \xymatrix{
(A,X) \ar^{(i,\iota)}[r]\ar^{(f,\vphi)}[d] & (B,Y) \ar^{(r,\rho)}[r]\ar^{(g,\psi)}[d] & (A,X) \ar^{(\vphi,f)}[d] \\
(A',X') \ar^{(i',\iota')}[r] & (B',Y') \ar^{(r',\rho')}[r] & (A',X') \\
}$$
in $\int_\M\F$. Then we get in particular a retract diagram of maps
$$ \xymatrix{
A \ar^{i}[r]\ar^{f}[d] & B \ar^{r}[r]\ar^{g}[d] & A \ar^{f}[d] \\
A' \ar^{i'}[r] & B' \ar^{r'}[r] & A' \\
}$$
in $\M$ and a retract diagram of maps
\begin{equation}\label{e:retract-2}
\xymatrix{
f_{!}X \ar[r]\ar^{\vphi}[d] & r'_{!}g_{!}Y \ar[r]\ar^{r_{!}'\psi}[d] & f_{!}X \ar^{\vphi}[d] \\
X' \ar[r] & r'_{!}Y' \ar[r] & X' \\
}
\end{equation}
in $\F(A')$. Now assume that $(g,\psi)$ is a (trivial) cofibration. Then $g$ is a (trivial) cofibration in $\M$ and $\psi$ is a (trivial) cofibration in $\F(B')$. In this case $r'_{!}\psi$ will be a (trivial) cofibration in $\F(A)$. Hence both $f$ and $\vphi$ will be (trivial) cofibrations in $\M$ and $\F(A')$ respectively and so $(f,\vphi)$ will be a (trivial) cofibration in $\int_\M \F$. This shows that $\Cof$ and $\W \cap \Cof$ are closed under retracts.

To show that $\Fib$ and $\W \cap \Fib$ are closed under retracts we observe that the diagram~\ref{e:retract-2} also induces a retract diagram of the form
$$ \xymatrix{
X \ar[r]\ar^{\vphi^{\ad}}[d] & i^*Y \ar[r]\ar^{i^*\psi^{\ad}}[d] & X \ar^{\vphi^{\ad}}[d] \\
f^*X' \ar[r] & i^*g^*Y' \ar[r] & f^*X' \\
}$$
in the model category $\F(A)$. A similar argument will now show that if $g$ is a (trivial) fibration in $\M$ and $\psi^{\ad}$ is a (trivial) fibration in $\F(B)$ then $f$ is a (trivial) fibration in $\M$ and $\vphi^{\ad}$ is a (trivial) fibration in $\F(A)$. This shows that $\Fib$ and $\W \cap \Fib$ are closed under retracts.
\end{proof}

\end{enumerate}

\subsection{Comparison with the $\infty$-categorical Grothendieck construction}\label{ss:infinity-cat pov}

Let $\Set_\Del$ denote the category of simplicial sets and let $\Set^+_\Del$ denote the category of marked simplicial sets. The category $\Set_\Del$ can be endowed with the Joyal model structure (see~\cite[Theorem $2.2.5.1$]{Lur09}) yielding a model for the theory of $\infty$-categories. Similarly, the category $\Set^+_\Del$ can be endowed with the coCartesian model structure (see~\cite[Remark $3.1.3.9$]{Lur09}) yielding another model for the theory of $\infty$-categories. The adjunction
$$ \xymatrix{
\Set_\Del \ar[r]<1ex>^{(\bullet)^{\flat}} & \Set^+_\Del \ar[l]<1ex>_{\upvdash}^{U} \\
}$$
is a Quillen equivalence, where $X^{\flat} = (X,s_0(X_0))$ is the minimal marked simplicial set on $X$ and $U(X,M) = X$. Given a model category $\M$, we will denote by 
$$ \M_\infty = U\left(\N\left(\M^{\cof}\right),\N\left(\W\cap \M^{\cof}\right)\right) $$ 
the underlying simplicial set of the fibrant replacement of the marked simplicial set $\left(\N\left(\M^{\cof}\right),\N\left(\W \cap \M^{\cof}\right)\right)$. Here, $\M^{\cof} \subseteq \M$ denotes the full subcategory of cofibrant objects. Following Lurie (see~\cite[Definition $1.3.4.15$]{Lur11}), we will refer to $\M_\infty$ as the underlying $\infty$-category of $\M$. In the case of $\Set^+_\Del$, we will also denote by $\Cat_\infty \x{\df}{=} \left(\Set^+_\Del\right)_\infty$ the underlying $\infty$-category of $\infty$-categories. 

Now let $\F:\M \lrar \ModCat$ be a proper relative functor. By restricting attention to the left Quillen functors one obtains a relative functor
$$ \F^{\cof}: \M^{\cof}\lrar \RelCat $$
given by $\F^{\cof}(A) = (\F(A))^{\cof}$. Since $\F$ is relative we get in particular that $\F^{\cof}$ sends weak equivalences in $\M^{\cof}$ to Dwyer-Kan equivalences of relative categories. Composing with the nerve functor we obtain a functor
$$ \N \F^{\cof} :\M^{\cof} \lrar \Set^+_\Del = \left(\Set^+_\Del\right)^{\cof} $$
which sends $\W^{\cof}$ to weak equivalences. We hence obtain a map of $\infty$-categories
$$ \F_\infty: \M_\infty \lrar \Cat_\infty $$

\begin{rem}
The construction of an $\infty$-categorical map $\F_\infty$
underlying $\F$ depends crucially on the fact that $\F$ is \textbf{relative}. For non-relative functors there is often no well-defined way to associate a map of the form $\M_\infty \lrar \Cat_\infty$. A concrete example illustrating this can be found in Example~\ref{e:non-relative} below. 
\end{rem}

According to Lurie's $\infty$-categorical Grothendieck construction, given by the unstraightening functor (see~\cite[Theorem $3.2.0.1$]{Lur09}), there exists an equivalence of $\infty$-categories between $\Fun(\M_\infty,\Cat_\infty)$ and the $\infty$-category of coCartesian fibrations $X \lrar \M_\infty$. We will denote by 
$$ \int_{\M_\infty}\F_\infty \lrar \M_\infty $$
the coCartesian fibration associated with $\F_\infty$ by the aforementioned equivalence. The purpose of this section is to relate the underlying $\infty$-category of the integral model structure on $\int_\M \F$, constructed in the previous subsection, to the $\infty$-category $\int_{\M_\infty} \F_\infty$. We first observe that the natural projection
$$ \int_\M \F \lrar \M $$
induces a map of $\infty$-categories
$$ \left(\int_\M\F\right)_\infty \lrar \M_\infty $$
We then have the following Proposition:
\begin{pro}
Let $\F: \M \lrar \ModCat$ be a proper relative functor and let
$$ \F_\infty: \M_\infty \lrar \Cat_\infty $$
be the associated $\infty$-functor as above. Then there is a natural equivalence of $\infty$-categories over $\M_\infty$
$$ \xymatrix{
\left(\int_\M\F\right)_\infty \ar^{\simeq}[rr] \ar[dr] && \int_{\M_\infty} \F_\infty \ar[dl] \\
& \M_\infty & \\
}$$
\end{pro}
\begin{proof}
Let $\W,\V$ be the classes of weak equivalences of $\M$ and of the integral model structure $\int_\M \F$ (respectively). It is straightforward to verify that the map $$\left(\N\left(\int_\M \F\right)^{cof}, \N\V^{cof}\right)\lrar \left(\N\M^{\cof},\N\W^{cof}\right)$$ is a marked cocartesian fibration in the sense of \cite[Definition $3.1.1$]{Hin} (the marking on $\N\M^{\cof}$ is saturated since $\M$ is a model category). Hence the first part of \cite[Proposition $3.1.4$]{Hin} implies that $\left(\int_\M\F\right)_\infty\lrar \M_\infty$ is a coCartesian map. The second part of \cite[Proposition $3.1.4$]{Hin} implies the equivalence of $\infty$-categories (over $\M_\infty$) $$\left(\int_\M\F\right)_\infty \overset{\simeq}{\lrar} \int_{\M_\infty} \F_\infty .$$ 
\end{proof}

In light of the main result of Gepner, Haugseng and Nikolaus (\cite[Theorem 1.1]{GHN}), we have the following immediate corollary:
\begin{cor}
Let $\F: \M \lrar \ModCat$ be a proper relative functor and let
$ \F_\infty: \M_\infty \lrar \Cat_\infty $
be the associated $\infty$-functor as above. Then the integral model category $\int_\M\F$ is a model for the lax colimit of $\F_\infty$.
\end{cor}

\section{Functoriality and Invariance}\label{s:invariance}

In this section we will discuss the behaviour of the integral model structure under various constructions. We begin by observing that if $\M$ is a category then a morphism in the $(2,1)$-category $\Fun(\M,\ModCat)$ are given by pseudo-natural transformations $(\sig,\tau):\F \Rightarrow \G$ such that for every $A \in \M$ the adjunction
$$ \xymatrix{
\F(A) \ar[r]<1ex>^{\L} & \G(A) \ar[l]<1ex>_{\upvdash}^{\R} \\
}$$
is a Quillen adjunction. We will call such pseudo-natural transformations \textbf{Quillen transformations}. Now let $\M$ be a model category. We will denote by 
$$ \Fun_{\PR}(\M,\ModCat) \subseteq \Fun(\M,\ModCat) $$ 
the full $(2,1)$-category spanned by the proper relative functors. We recall the following observation for future reference.

\begin{obs}\label{o:functoriality}
Let $\M$ be a model category. If $(\sig,\tau): \F \Rightarrow \G$ is a Quillen transformation then the induced adjunction
$$ \xymatrix{
\int_\M\F \ar[dr]\ar[rr]<1ex>^{\sig_*} && \int_\M\G \ar[ll]<1ex>_{\upvdash}^{\tau_*} \ar[dl] \\
& \M & \\
}$$
is simply given by $\sig_*(A,X) = (A,\sig(X))$ and $\tau_*(A,X) = (A,\tau_*(X))$. It is then clear that under the respective integral model structures the adjunction $(\sig_*,\tau_*)$ is a Quillen adjunction. 
\end{obs}

\subsection{Base change}
Recall the notions of \textbf{left and right morphisms} discussed in \S\ref{d:left-right}. We wish to address the analogous setting for diagrams of model categories.

\begin{defn}
Let $\M,\N$ be model categories and 
$$ \xymatrix{
\M \ar[dr]_\F\ar[rr]<1.3ex>^{\L} & & \N\ar[dl]^\G\ar[ll]<0.7ex>^{\R}_\upvdash \\ &  \ModCat & \\}$$
a (not necessarily commutative) diagram such that the horizontal pair is a Quillen adjunction and  $\F,\G$ are proper relative functors. We will say that a left morphism $\F \Rightarrow \G \circ \L$ is a \textbf{left Quillen morphism} if the associated adjunctions
$$\xymatrix{
\Sig^L_A:\F(A) \ar[r]<1ex> & \G(\L(A)):\Sig_A^R\ar[l]<1ex>_(0.55){\upvdash}.}$$
are Quillen adjunction. Similarly we define \textbf{right Quillen morphisms}.
\end{defn}

\begin{define}\label{d:left-right-equiv}
Let $\M,\N,\F,\G$ be as above. We will say that a left Quillen morphism
$$\xymatrix{
\Sig^L_A:\F(A) \ar[r]<1ex> & \G(\L(A)):\Sig^R_A\ar[l]<1ex>_(0.55){\upvdash}.}$$
indexed by $A \in \M$ is a \textbf{left Quillen equivalence} if $\Sig^L_A \dashv \Sig^R_A$ is a Quillen equivalence for every cofibrant $A \in \M$. Similarly, we will say that a right Quillen morphism
$$\xymatrix{
\Theta^L_B:\F(\R(B)) \ar[r]<1ex> & \G(B) :\Theta^R_B\ar[l]<1ex>_(0.45){\upvdash}}$$
indexed by $B \in \N$ is a \textbf{right Quillen equivalence} if $\Theta^L_B \dashv \Theta^R_B$ is a Quillen equivalence for every fibrant $B \in \N$
\end{define}

\begin{thm}\label{qa}
Let $\M,\N$ be model categories and
$$ \xymatrix{
\M \ar[dr]_\F\ar[rr]<1.2ex>^{\L} & & \N\ar[dl]^\G\ar[ll]<0.6ex>^{\R}_\upvdash \\ &  \ModCat & \\}$$
a diagram as above in which the horizontal pair is a Quillen adjunction and $\F,\G$ are proper relative functors. Let $\F \Rightarrow \G \circ \L$ be a left Quillen morphism given by a compatible family of adjunctions $\left(\Sig^L_A, \Sig^R_A\right)_{A \in \M}$. Then the induced adjunction
$$
\xymatrix{\Phi^L:\int_{\M}\F\ar[r]<1ex> & \int_{\N}\G:\Phi^R\ar[l]<1ex>_(0.5){\upvdash}.}
$$
is a Quillen adjunction. Furthermore, if the left Quillen morphism is a left Quillen equivalence then $\left(\Phi^L,\Phi^R\right)$ is a Quillen equivalence. The same result holds for the adjunction induced by a right Quillen morphism (see Remark~\ref{r:dual}).
\end{thm}
\begin{proof}
Using Lemma~\ref{characterization} it is straightforward to verify that $\Phi^L$ preserves cofibrations and trivial cofibrations and hence is a left Quillen functor.

Now assume that $\Sig^L_A \dashv \Sig^R_A$ is a Quillen equivalence whenever $A$ is cofibrant. Let $(A,X) \in \int_\M \F$ be a cofibrant object, $(B,Y) \in \int_\N \G$ a fibrant object and consider a 
map
$$ (f,\vphi): (A,X) \lrar \left(\R(B),\Sig^R_{\R(B)}(\eps_{!}Y)\right) $$
where $f: A \lrar \R(B)$ is a morphism in $\M$ and 
$$ \vphi: f_{!}X \lrar \Sig^R_B(\eps_{!}Y) $$ 
is a morphism in $\F(\R(B))$. Since $X$ is cofibrant in $\F(A)$ we get $(f,\vphi)$ is a weak equivalence if and only if $f$ is a weak equivalence in $\M$ and $\vphi$ is a weak equivalence in $\F(\R(B))$. Since $\F$ is relative this is equivalent to $f$ being a weak equivalence and 
$$ \vphi^{\ad}: X \lrar f^*\left(\Sig^R_{\R(B)}(\eps^*Y)\right) \;\cong\; \Sig^R_A\left(\L(f)^*(\eps^*Y)\right) \;\cong\; \Sig^R_A (f^{\ad})^*Y $$
being a weak equivalence, where the first isomorphism is given by the structure of $\Sig^R$ as a pseudo-natural transformation and $f^{\ad} = \eps \circ \L(f)$ is the adjoint of $f$. According to Remark~\ref{r:explicit} the adjoint morphism to $(f,\vphi)$ is given by the map
$$ (f^{\ad},\psi): (\L(A),\Sig^L_A(X)) \lrar (B,Y) $$
where
$$ \psi: (f^{\ad})_{!}\Sig^L_A(X) \lrar Y $$ 
is the adjoint of $\vphi^{\ad}$ with respect to $\left(f^{\ad}\right)_{!} \circ \Sig^L_A   \dashv  \Sig^R_A \circ \left(f^{\ad}\right)^*$. Since $\L \dashv \R$ is a Quillen equivalence we see that $f$ is a weak equivalence if and only if $f^{\ad}$ is a weak equivalence. Furthermore, in this case the adjunction $\left(f^{\ad}\right)_{!} \circ \Sig^L_A   \dashv  \Sig^R_A \circ \left(f^{\ad}\right)^*$ is a Quillen equivalence by our assumptions and so $\vphi^{\ad}$ is a weak equivalence if and only if $\psi$ is a weak equivalence. This shows that $\Phi^L \dashv \Phi^R$ is a Quillen equivalence. The proof for the adjunction induced by a right Quillen morphism is completely analogous.
\end{proof}

\begin{rem}
The invariance property established in Theorem~\ref{qa} depends crucially on the fact that $\F$ is \textbf{relative}. For non-relative functors this invariance may fail drastically, as illustrated in Example~\ref{e:non-relative} below.
\end{rem}

\begin{example}\label{e:non-relative}
Let $\M=\{\emptyset\overset{\sim}{\lrar} \ast\}$ be the category with two objects and one non-identity morphism. It is straightforward to check that $\M$ has all small limits and colimits. We may endow $\M$ with a model structure in which every morphism is a weak equivalence and a cofibration, and only the isomorphisms are fibrations. A functor $\F:\M \lrar \ModCat$ can then be described as a Quillen adjunction
$$ \xymatrix{
\F(\emptyset)\ar[r]<1ex>^{\L_{\F}} & \F(\ast) \ar[l]<1ex>^{\R_{\F}}_{\upvdash}}
$$
Note that $\F$ as above is relative precisely when the Quillen adjunction $\L_{\F} \dashv \L_{\R}$ is a Quillen equivalence. Now assume that $\F$ is not necessarily relative, but is such that the choice of weak equivalences, cofibrations and fibrations appearing in Definition~\ref{d:model} endows $\int_\M \F$ with a model structure. This happens, for example, whenever $\F(\emptyset)$ is the trivial model category with one object and no non-identity morphisms. The inclusion $\{\ast\} \subseteq \M$ (where the left hand side carries a trivial model structure) is a right Quillen functor and induces a Quillen adjunction
$$ \xymatrix{\int_\M\F\ar[r]<1ex>^-{\gam_\ast} & \int_{\{\ast\}} \F = \F(\ast) \ar[l]<1ex>^-{\iota_\ast}_-{\upvdash}}$$
where $\gam_\ast(\emptyset,X) = \L_\F(X)$, $\gam_\ast(\ast,X) = X$ and $\iota_\ast(X) = (\ast,X)$. The counit of this adjunction is an isomorphism and its derived unit is a weak equivalence by Definition~\ref{d:model}. We hence see that $\iota_\ast \dashv \gam_\ast$ is a Quillen equivalence. On the other hand, the inclusion $\{\emptyset\} \subseteq \M$ is a left Quillen functor, inducing a Quillen adjunction
$$ \xymatrix{\F(\emptyset) = \int_{\{\emptyset\}} \F \ar[r]<1ex>^-{\iota_\emptyset} & \int_\M\F \ar[l]<1ex>^-{\gam_\emptyset}_-{\upvdash}}$$
where $\gam_\emptyset(\emptyset,X) = X$, $\gam_\emptyset(\ast,X) = \R_{\F}(X)$ and $\iota_\emptyset(X) = (\emptyset,X)$. One can then verify that the composite
$$ \xymatrix{\F(\emptyset)\ar[r]<1ex>^{\iota_\emptyset} & \int_\M\F \ar[l]<1ex>^{\gam_\emptyset}_{\upvdash} \ar[r]<1ex>^{\gam_\ast} & \F(\ast) \ar[l]<1ex>^{\iota_\ast}_{\upvdash}}$$
coincides with $\L_\F \dashv \R_\F$. We hence see that if $\F$ is \textbf{not relative} then $\gam_{\emptyset} \dashv \iota_{\emptyset}$ cannot be a Quillen equivalence, despite the fact that the inclusion $\{\emptyset\} \subseteq \M$ is a Quillen equivalence.
\end{example}

\subsection{A Fubini Theorem}\label{ss:fubini}

Let $\M,\N$ be categories and $\F:\M\times \N\lrar \Cat$ a functor. For each $A\in\M$ we have a functor $\F^A:\{A\}\times\N\lrar \Cat$ and for each $B\in\N$ we have a functor $\F_B:\M\times\{B\}\lrar \Cat$. It is immediate to notice that we have an \textbf{isomorphism} of categories
\begin{equation}\label{e:discrete Fubini}\displaystyle\mathop{\int}_{(A,B)\in\M\times \N}\F(A,B)\cong \mathop{\int}_{A\in \M}\mathop{\int}_{B\in \N}\F^A(B)\cong\mathop{\int}_{B\in \N}\mathop{\int}_{A\in \M}\F_B(A) \end{equation}
The purpose of this section is to extend the above results to the setting of model categories. Let $\M,\N$ be model categories and $\F:\M\times \N\lrar \ModCat$ a functor where $\M\times \N$ is endowed with the product model structure. Since for any $(f,g):(A,B)\lrar (A',B')$ in $\M\times \N$, $(f,g)=(f,\Id_B)\circ (\Id_{A'},g)$ we see that $\F$ is proper relative iff for each $A\in\M$ and each $B\in\N$, $F^A$ and $F_B$ are proper and relative. We now have the following:

\begin{pro}["Fubini Theorem"]\label{p:fubini}
Let $\M,\N$ be model categories and \\$\F:\M\times\N\lrar \ModCat$ a functor which is proper and relative. Then the functors  $$\displaystyle\mathop{\int}_{B \in \N}\F^{(-)}(B):\M\lrar \ModCat$$ and 
$$\displaystyle\mathop{\int}_{A \in \M}\F_{(-)}(A):\N\lrar \ModCat$$ are proper and relative and we have natural isomorphisms of model categories 

$$\displaystyle\mathop{\int}_{(A,B) \in \M\times \N}\F(A,B)\cong \mathop{\int}_{A\in \M}\left(\mathop{\int}_{B \in \N}\F^A(B)\right )\cong\mathop{\int}_{B\in \N}\left(\mathop{\int}_{A \in \M}\F_B(A)\right)$$
\end{pro}
\begin{proof}
We will prove that $\int_{B \in \N}\F^{(-)}(B)$ is proper and relative. The proof for $\int_{A \in \M}\F_{(-)}(A)$ is completely analogous. The fact that $\int_{B \in \N}\F^{(-)}(B)$ is relative follows from the invariance of the integral model structure (Theorem \ref{qa}). To see that $\int_{\N}\F^{(-)}$ is proper, let $f: A \lrar A'$ be a trivial cofibration and let
$$ (\Id_A,g,\vphi):(A,B,X)\lrar (A,B',X') $$ 
be a weak equivalence in $\mathop{\int}_{B \in \N}\F^A(B)$ where $g: B \lrar B'$ is a weak equivalence in $\N$ and
$$ \vphi: (\Id_A,g)_!X \lrar X' $$
is a morphism in $\F^A(B')$ such that the composite
$$ (\Id_A,g)_!(X^{\cof}) \lrar (\Id_A,g)_!X \lrar X' $$
is a weak equivalence in $\F^{A'}(B)$. Then $f_!(\Id_A,g,\vphi)$ can be identified with the map
$$ (\Id_{A'},g,\psi):(A',B,(f,\Id_B)_!(X))\lrar (A',B',(f,\Id_{B'})_!(X'))$$ 
in $\mathop{\int}_{B \in \N}\F^A(B)$ where 
$$ \psi: (\Id_{A'},g)_!(f,\Id_B)_!X = (f,\Id_{B'})_!(\Id_{A},g)_!X \lrar (f,\Id_{B'})_!X' $$
is given by $(f,\Id_{B'})_!\vphi$.

We then need to verify that the composite
$$ (f,\Id_{B'})_!(\Id_{A},g)_!X^{\cof} \lrar (f,\Id_{B'})_!(\Id_{A},g)_!X \lrar (f,\Id_{B'})_!X' $$

is a weal equivalence in $\F(A',B')$. But this now follows from the fact that $\F$ is proper and $(f,\Id_{B'}):(A,B') \lrar (A',B')$ is a trivial cofibration in $\M \times \N$.

The isomorphism of \ref{e:discrete Fubini} together with Lemma~\ref{characterization} now easily implies that all the above-mentioned model structures coincide.
\end{proof}

\section{Model fibrations}\label{s:model-fib}

Let $\M$ be a category equipped with three subcategories $\W_\M,\Cof_\M,\Fib_\M \subseteq \M$ which contain all objects. We shall refer to such objects as \textbf{pre-model categories}. Given a pre-model category $\M$, we will still refer to morphisms in $\W_\M$, $\Cof_\M$, $\Fib_\M$, $\Cof_\M \cap \W_\M$ and $\Fib_\M \cap \W_\M$ as weak equivalences, cofibrations, fibrations, trivial cofibrations and trivial fibrations respectively. A morphism of pre-model categories is an adjunction
$$\xymatrix{\M\ar^{\L}[r]<1ex> &N\ar^{\R}[l]<1ex>_{\upvdash}}$$
such that $\L$ preserves cofibrations and trivial cofibrations and $\R$ preserves fibrations and trivial fibrations. We will refer to such adjunctions as \textbf{Quillen adjunctions}. 

We now observe that the notion of a \textbf{proper relative functor} (see Definitions~\ref{d:relative} and~\ref{d:proper}) can be extended to the case where the domain is a pre-model category \textbf{verbatim}. We will denote by 
$$ \Fun_{\PR}(\M,\ModCat) \subseteq \Fun(\M,\ModCat) $$ 
the full $(2,1)$-subcategory spanned by proper relative functors. 

Our goal in this section is to understand the Grothendieck construction of a proper relative functor $\F: \M \lrar \ModCat$ in the case where $\M$ is a pre-model category. We start by formulating a relative counterpart of the model category axioms. For this we will need a relative counterpart of the notion of a weak factorization system.

\begin{defn}\label{d:rel-wfs}
Let $\M, \N$ be categories, each equipped with two classes of maps $\L_\M,\R_\M \subseteq \M^{\Del^1}$ and $\L_\N,\R_\N \subseteq \N^{\Del^1}$ containing all the identities. Let $\pi:\N\lrar \M$ be a functor such that $\pi(\L_{\N}) \subseteq \L_\M$ and $\pi(\R_\N) \subseteq \R_\M$. We will say that $(\L_\N,\L_\M)$ constitute a \textbf{$\pi$-weak factorization systems relative to $(\L_\M,\R_\M)$} if the following holds:
\begin{enumerate}
\item
$\L_{\N}$ (resp. $\R_{\N}$) contains any retract $f$ of a morphism in $\L_{\N}$ (resp. $\R_{\N}$) provided that $\pi(f)$ is contained in $\L_\M$ (resp. $\R_\M$).
\item
For every morphism $\vphi: X \lrar Y$ in $\N$ and every factorization of $\pi\vphi$ as $\pi \vphi = g \circ h$ such that $h \in \L_\M$ and $g \in \R_\M$ there exists a factorization of $\vphi$ as $\vphi = \psi \circ \eta$ such that $\eta \in \L_\N,\psi \in \R_\N$ and such that $\pi\psi = g$ and $\pi\eta = h$.
\item
For every square in $\N$ of the form
$$ \xymatrix{
X \ar[r]\ar_{\psi}[d] & Z \ar^{\eta}[d] \\
Y \ar[r] & W \\
}$$
such that $\psi \in \L_\N$ and $\eta \in \R_\N$ and for every dashed lift
$$ \xymatrix{
\pi(X) \ar[r]\ar_{\pi\psi}[d] & \pi(Z) \ar^{\pi\eta}[d] \\
\pi(Y) \ar[r]\ar@{-->}^{u}[ur] & \pi(W) \\
}$$
there exists a dashed lift
$$ \xymatrix{
X \ar[r]\ar_{\psi}[d] & Z \ar^{\eta}[d] \\
Y \ar[r]\ar@{-->}^{\gam}[ur] & W \\
}$$
such that $\pi\gam = u$.
\end{enumerate}
\end{defn}

\begin{lem}\label{l:wfs-comp}
Let $\M,\M',\M''$ be three categories with corresponding pairs of classes of morphisms $\L_\M,\R_\M \subseteq \left(\M\right) ^{\Del^1}$, $\L_{\M'},\R_{\M'} \subseteq (\M')^{\Del^1}$ and $\L_{\M''},\R_{\M''} \subseteq (\M'')^{\Del^1}$, containing all the identities. Let $\pi: \M \lrar \M'$ and $\pi': \M' \lrar \M''$ be functors which preserve these classes of morphisms. If $(\L_\M,\R_\M)$ is a $\pi$-weak factorization system with respect to $(\L_{\M'},\R_{\M'})$ and $(\L_{\M'},\R_{\M'})$ is a $\pi'$-weak factorization system with respect to $(\L_{\M''},\R_{\M''})$ then $(\L_\M,\R_\M)$ is a $(\pi' \circ \pi)$-weak factorization system with respect to $(\L_{\M''},\R_{\M''})$.
\end{lem}
\begin{proof}
We will prove that $(\L_\M,\R_\M)$ satisfy conditions $(1)-(3)$ of Definition~\ref{d:rel-wfs}.
\begin{enumerate}
\item
Let $f$ be a morphism in $\N$ which is a retract of a morphism in $\L_\M$ such that $\pi'(\pi(f))$ is contained in $\L_{\M''}$. Since $\pi(f)$ is also a retract of a morphism in $\L_{\M'}$ we may conclude from property $(1)$ for $\pi'$ that $\pi(f)$ is in $\L_{\M'}$. Using property $(1)$ for $\pi$ it then follows that $f$ is in $\L_\M$ as desired. The proof for the case where $f$ is a retract of a morphism in $\R_{\M}$ is the same.
\item
Let $\vphi: X \lrar Y$ be a morphism in $\N$ and let $\pi'\pi\vphi = g \circ h$ be a factorization such that $h \in \L_{\M''}$ and $g \in \R_{\M''}$. By property $(2)$ for $\pi'$ there exists a factorization of $\pi\vphi$ as $\pi\vphi = \psi \circ \eta$ such that $\eta \in \L_{\M'},\psi \in \R_{\M'}$ and such that $\pi'\psi = g$ and $\pi'\eta = h$. Using property $(2)$ for $\pi$ we then deduce that there exists a factorization of $\vphi = G \circ H$ such that $H \in \L_{\M},G \in \R_{\M}$ and such that $\pi G = \psi$ and $\pi H = \eta$. It follows that $\pi'\pi G = g$ and $\pi'\pi H = h$.
\item
Let
$$ \xymatrix{
X \ar[r]\ar_{\psi}[d] & Z \ar^{\eta}[d] \\
Y \ar[r] & W \\
}$$
be a diagram in $\M$ such that $\psi \in \L_\M$ and $\eta \in \R_\M$ and equipped with a dashed lift
$$ \xymatrix{
\pi'\pi(X) \ar[r]\ar_{\pi'\pi\psi}[d] & \pi'\pi(Z) \ar^{\pi'\pi\eta}[d] \\
\pi'\pi(Y) \ar[r]\ar@{-->}^{u}[ur] & \pi'\pi(W) \\
}$$
By property $(3)$ for $\pi'$ there exists a dashed lift
$$ \xymatrix{
\pi(X) \ar[r]\ar_{\pi\psi}[d] & \pi(Z) \ar^{\pi\eta}[d] \\
\pi(Y) \ar[r]\ar@{-->}^{\gam}[ur] & \pi(W) \\
}$$
such that $\pi'\gam = u$. By property $(3)$ for $\pi$ there exists a dashed lift
$$ \xymatrix{
X \ar[r]\ar_{\psi}[d] & Z \ar^{\eta}[d] \\
Y \ar[r]\ar@{-->}^{U}[ur] & W \\
}$$
such that $\pi U = \gam$ and so $\pi'\pi U = u$.

\end{enumerate}
\end{proof}

We are now ready to define the relative analogue of the notion of a model category.
\begin{defn}\label{d:relative-model}
Let $\M,\N$ be two pre-model categories. We will say that a functor $\pi: \N \lrar \M$ exhibits $\N$ as a \textbf{model category relative to $\M$} if the following conditions are satisfied:
\begin{enumerate}
\item
$\pi$ is bicomplete.
\item
Let $f: X \lrar Y$ and $g: Y \lrar Z$ be morphisms in $\N$. If two of $f,g,g\circ f$ are in $\W_\N$ and if the image of the third is in $\W_\M$ then the third is in $\W_\N$.
\item
$(\Cof_\N\cap \W_\N,\Fib_\N)$ and $(\Cof_\N,\Fib_\N\cap \W_\N)$ are $\pi$-weak factorization systems relative to $(\Cof_\M \cap \W_\M,\Fib_\M)$ and $(\Cof_\M,\Fib_\M \cap \W_\M)$ respectively.
\end{enumerate}
In this case we will also say that $\pi$ is a \textbf{relative model category}.
\end{defn}

\begin{rem}\label{r:absolute}
The terminology of relative model category can be justified by the fact that a pre-model category $\M$ is a model category precisely when the terminal map $\M \lrar \ast$ is a relative model category.
\end{rem}

\begin{pro}\label{p:rel-model-comp}
Let $\pi: \M \lrar \M'$ and $\pi': \M' \lrar \M''$ be relative model categories. Then $\pi' \circ \pi: \M \lrar \M''$ is a relative model category.
\end{pro}
\begin{proof}
By Remark~\ref{r:cocompose} we see that $\pi' \circ \pi$ will satisfy $(1)$ of Definition~\ref{d:relative-model}. It is straightforward to check that it will satisfy $(2)$ as well. Finally, property $(3)$ for $\pi'\circ \pi$ follows from Lemma~\ref{l:wfs-comp}
\end{proof}

Combining Proposition~\ref{p:rel-model-comp} and Remark~\ref{r:absolute} we obtain the following corollary:
\begin{cor}
Let $\pi: \N \lrar \M$ be a relative model category such that $\M$ is a model category. Then $\N$ is a model category.
\end{cor}
\begin{proof}
A pre-model category $\M$ is a model category if and only if $\M\lrar \{*\}$ is a model category.
\end{proof}

Now let $\pi:\N \lrar \M$ be a relative model category. Since $\pi$ is bicomplete we deduce that for each $A \in \M$ the fiber $\N \times_\M \{A\}$ is bicomplete. We will denote by $\emptyset_A, \ast_A \in \N \times_\M \{A\}$ the initial and terminal objects of $\N \times_\M \{A\}$, respectively. We will say that an object $X \in \N$ is $\pi$-cofibrant if the unique map $\emptyset_{\pi(X)} \lrar X$ covering $\Id_{\pi(X)}$ is in $\Cof_\N$. Similarly, we will say that an object $X \in \N$ is $\pi$-fibrant if the unique map $X \lrar \ast_{\pi(X)}$ covering $\Id_{\pi(X)}$ is in $\Fib_\N$.

\begin{defn}\label{d:model-fib}
Let $\pi: \N \lrar \M$ be a relative model category. We will say that $\pi$ is a \textbf{model fibration} if the following conditions are satisfied:
\begin{enumerate}
\item
The underlying functor of $\pi$ is a biCartesian fibration.
\item
If $f: X \lrar Y$ is a $\pi$-coCartesian morphism in $\N$ such that $X$ is $\pi$-cofibrant and $\pi(f) \in \W_\M$ then $f \in \W_\N$.
\item
If $f: X \lrar Y$ is a $\pi$-Cartesian morphism in $\N$ such that $Y$ is $\pi$-fibrant and $\pi(f)$ is in $\W_\M$ then $f \in \W_\N$.
\end{enumerate}
A \textbf{morphism} of model fibrations over $\M$ is a Quillen adjunction
$$\xymatrix{
\N\ar[dr]_{\pi} \ar[rr]<1ex>^{\Phi} && \N'\ar[dl]^{\pi'}\ar^{\Psi}[ll]<1ex>_{\upvdash}\\& \M &}$$
over $\M$ such that $\Phi$ preserves coCartesian morphisms and $\Psi$ preserves Cartesian morphisms. The $2$-morphisms are given by psuedo-transformations of adjunctions. We will denote the resulting $(2,1)$-category by $\ModFib(\M)$.
\end{defn}

\begin{thm}\label{t:mod-fib}
Let $\M$ be a pre-model category and $\F: \M \lrar \ModCat$ be a proper relative functor. Then
$$ \pi:\int_\M \F \lrar \M $$
is a model fibration.
\end{thm}
\begin{proof}
The fact that $\pi$ is a relative model category follows by examining the proof of Theorem~\ref{model structure}. Now $\pi$ is clearly a biCartesian fibration and properties (2) and (3) of Definition~\ref{d:model-fib} are a direct consequence of the definition of weak equivalences in the integral model structure.
\end{proof}

Let $\M$ be a pre-model category. Observation~\ref{o:functoriality} extends to the case of pre-model categories verbatim. In particular, the association $\F \mapsto \int_\M\F$ determines a functor of $(2,1)$-categories
\begin{equation}\label{e:groth}
\int_\M : \Fun_{\PR}(\M,\ModCat) \lrar  \ModFib(\M)
\end{equation}

Our purpose in this section is to prove the following theorem, which is a model categorical analogue of Grothendieck's classical correspondence:
\begin{thm}\label{t:equiv}
Let $\M$ be a pre-model category. The functor $\int_\M$ above is an equivalence of $(2,1)$-categories.
\end{thm}

In order to prove Theorem~\ref{t:equiv} we will need several lemmas.

\begin{lem}\label{l:car-lem1}
Let $\M,\N$ be pre-model categories and let $\pi: \N \lrar \M$ be a biCartesian fibration.
\begin{enumerate}
\item
If $\pi$ is right Quillen and $\vphi: X \lrar Y$ is a $\pi$-coCartesian morphism in $\N$ such that $\pi\vphi$ is a (trivial) cofibration in $\M$ then $\vphi$ is a (trivial) cofibration in $\N$.
\item
If $\pi$ is left Quillen and $\vphi: X \lrar Y$ is a $\pi$-Cartesian morphism in $\N$ such that $\pi\vphi$ is a (trivial) fibration in $\M$ then $\vphi$ is a (trivial) fibration in $\N$.
\end{enumerate}
\end{lem}
\begin{proof}
We shall prove assertion (1) above. The proof of assertion (2) is completely analogous. Assume that $\pi\vphi$ is a cofibration. We need to show that $\vphi$ has the left lifting property with respect to trivial fibration.

Let
\begin{equation}\label{e:square3}
\xymatrix{
X \ar_{\vphi}[d]\ar^{\psi}[r] & X' \ar^{\rho}[d] \\
Y \ar_{\eta}[r] & Y' \\
}
\end{equation}
be a commutative diagram in $\N$ such that $\rho$ is a trivial fibration. Since $\pi$ is right Quillen the map $\pi\rho$ is a trivial fibration. Hence the projected square
$$\xymatrix{
\pi(X) \ar_{\pi\vphi}[d]\ar^{\pi\psi}[r] & \pi(X') \ar^{\pi \rho}[d] \\
\pi(Y) \ar_{\pi\eta}[r]\ar@{-->}^{u}[ur] & \pi(Y') \\
}$$
admits a dashed lift $u$. Since $\vphi$ is $\pi$-coCartesian there exists a dashed lift
$$\xymatrix{
X \ar_{\vphi}[d]\ar^{\psi}[r] & X' \ar^{\rho}[d] \\
Y \ar_{\eta}[r]\ar@{-->}^{\gam}[ur] & Y' \\
}$$
such that $\pi\gam = u$ and $\gam \circ \vphi =\psi$. From Remark~\ref{r:unique} we get that $\rho \circ \gam = \eta$ as well and so $\gam$ is a lift in the square~\ref{e:square3}. The case where $\pi\vphi$ is a trivial cofibration can be proven using the same argument by taking $\rho$ to be an arbitrary fibration.
\end{proof}

\begin{cor}
If $\pi: \N \lrar \M$ is a model fibration then $\pi$ is both a left and a right Quillen functor (of pre-model categories). Its left adjoint is the functor $A \mapsto \emptyset_A$ and its right adjoint is the functor $A \mapsto \ast_A$.
\end{cor}

\begin{cor}\label{c:necessary}
Let $\M$ be a pre-model category and let $\F: \M \lrar \ModCat$ be a functor. Assume that there exists a model structure on $\int_\M\F$ such that
\begin{enumerate}
\item
The map $\pi:\int_\M\F \lrar \M$ is a right (left) Quillen functor.
\item
For every $A \in \M$, a morphism $\vphi: X \lrar X'$ is a weak equivalence in $\F(A)$ if and only if $(\Id,\vphi):(A,X) \lrar (A,X')$ is a weak equivalence in $\int_\M\F$.
\end{enumerate}
Then $\F$ is left (right) proper.
\end{cor}
\begin{proof}
We will prove that if $\pi$ is right Quillen then $\F$ is left proper. The dual case is completely analogous. Let $f: A \lrar B$ be a trivial cofibration in $\M$ and $\vphi: X \lrar X'$ is a weak equivalence in $\F(A)$. Then we have a commutative square in $\int_\F \M$ of the form
$$ \xymatrix{
(A,X) \ar[r]\ar_{(\Id,\vphi)}[d] & (B, f_!X) \ar^{(\Id,f_!\vphi)}[d] \\
(A,X') \ar[r] & (B,f_!X') \\
}$$
where the horizontal edges are $\pi$-coCartesian. By assumption (2) above $(\Id,\vphi)$ is a weak equivalence in $\int_\M\F$. From Lemma~\ref{l:car-lem1} it follows that the horizontal maps are trivial cofibrations and so $(\Id,f_!\vphi)$ is a weak equivalence in $\int_\M\F$. By assumption (2) above we get that $f_!\vphi$ is a weak equivalence in $\F(B)$.
\end{proof}

\begin{lem}\label{l:car-lem2}
Let $\pi: \N \lrar \M$ be a model fibration. Let
\begin{equation}\label{e:coc}
\xymatrix{
X \ar^{\psi}[r]\ar_{\vphi}[d] & Y \ar^{\vphi'}[d] \\
X' \ar_{\eta}[r] & Y' \\
}
\end{equation}
be a commutative diagram in $\N$. Then
\begin{enumerate}
\item
If $\psi,\eta$ are $\pi$-coCartesian, $\vphi$ is a (trivial) cofibration in $\N$ and $\pi \vphi'$ is a (trivial) cofibration in $\M$ then $f'$ is a (trivial) cofibration in $\N$.
\item
If $\psi,\eta$ are $\pi$-Cartesian, $\vphi$ is a (trivial) fibration in $\N$ and $\pi \vphi'$ is a (trivial) fibration in $\M$ then $\vphi'$ is a (trivial) fibration in $\N$.

\end{enumerate}
\end{lem}
\begin{proof}
We shall prove assertion (1) above. The proof of assertion (2) is completely analogous.

Assume that $\vphi$ is a cofibration in $\N$ and $\pi \vphi'$ is a cofibration in $\M$. We shall show that $\vphi'$ has the left lifting property with respect to trivial fibrations in $\N$. Let

$$ \xymatrix{
X \ar^{\psi}[r]\ar_{\vphi}[d] & Y \ar^{\vphi'}[d]\ar^{\psi'}[r] & Z\ar^{\rho}[d] \\
X' \ar_{\eta}[r] & Y' \ar_{\eta'}[r] & Z' \\
}$$
be an extension of the diagram~\ref{e:coc} such that $\rho$ is a trivial fibration. The the right-most square in the projected diagram
\begin{equation}\label{e:rec-down}
\xymatrix{
\pi(X) \ar^{\pi \psi}[r]\ar_(0.4){\pi \vphi}[d] & \pi(Y) \ar[d]_(0.4){\pi \vphi'}\ar[r] & \pi(Z)\ar^(0.4){\pi \rho}[d] \\
\pi(X') \ar_{\pi \eta}[r]\ar@{-->}[urr] & \pi(Y') \ar@{..>}[ur]\ar[r] & \pi(Z') \\
}
\end{equation}
admits a dotted lift, yielding a dashed lift for the outer rectangle by composition. Using property (3) of Definition~\ref{d:model-fib} we can lift the dashed arrow of~\ref{e:rec-down} to a dashed arrow
\begin{equation}\label{e:rec-up}
\xymatrix{
X \ar^{\psi}[r]\ar_(0.3){\vphi}[d] & Y \ar_(0.3){\vphi'}[d]\ar^{\psi'}[r] & Z\ar^(0.3){\rho}[d] \\
X' \ar_{\eta}[r]\ar^(0.2){\nu}@{-->}[urr] & Y' \ar_{\xi}@{..>}[ur] \ar_{\eta'}[r] & Z' \\
}
\end{equation}
yielding a lift $\nu$ for the outer rectangle. Now the dotted arrow of~\ref{e:rec-down} is a factorization of $\pi \nu$ along $\pi \eta$. Since $\eta$ is $\pi$-coCartesian, there exists a unique dotted arrow $\xi$ in~\ref{e:rec-up} factorizing the dashed arrow of~\ref{e:rec-up} along $\eta$. By applying Remark~\ref{r:unique} to the $\pi$-coCartesian edge $\eta$ and to the pair $\rho\eta, \eta'$ we deduce that $\rho\xi=\eta'$. Similarly, by applying Remark~\ref{r:unique} to the $\pi$-coCartesian edge $\psi$ and the pair $\xi\vphi',\psi'$ we deduce that $\xi\vphi' = \psi'$. Hence $\xi$ is indeed a lift in the right-most square.

The case of $\vphi$ and $\pi \vphi'$ being trivial cofibrations can be proven using the same argument by taking $\rho$ to be an arbitrary fibration.

\end{proof}

\begin{proof}[Proof of Theorem~\ref{t:equiv}]
Let $\M$ be a pre-model category. We have a natural commutative diagram
$$\xymatrix{
\Fun_{\PR}(\M,\ModCat) \ar_{U}[d]\ar^(0.6){\int_\M}[r] & \ModFib(\M) \ar^{V}[d] \\
\Fun(\M,\AdjCat) \ar^(0.6){\simeq}[r] & \BiFib(\M) \\
}$$
where the vertical forgetful $(2,1)$-functors are faithful. It will hence suffice to verify the following claims
\begin{enumerate}
\item
The functor $\int_\M$ is essentially surjective.
\item
Let $\F,\G: \M \lrar \ModCat$ be proper relative functors and $(\sig,\tau): U(\F) \Rightarrow U(\G)$ be a pseudo-natural transformation such that the induced adjunction
$$ \xymatrix{
\int_\I\F \ar^{\sig_*}[rr]<1ex> \ar[dr] &&  \int_\I\G \ar^{\tau_*}_{\upvdash}[ll]<1ex> \ar[dl] \\
& \M & \\
}$$
is a Quillen adjunction. Then $(\sig,\tau)$ is a Quillen transformation.
\end{enumerate}

Let us begin by proving $(1)$. Let $\pi:\N \lrar \M$ be a model fibration. Then the underlying biCartesian fibration of $\pi$ determines a functor $\F: \M \lrar \AdjCat$. For each $A \in \M$, the category $\F(A)$ can be identified with the fiber $\N \times_{\M} \{A\}$ which inherits a natural structure of a model category by restricting $\W_\N,\Fib_\N$ and $\Cof_\N$ to $\N \times_{\M} \{A\}$ (that this is indeed a model structure can be seen by checking that $\N \times_{\M} \{A\} \lrar *$ is a relative model category). Furthermore, for each morphism $f: A \lrar B$ the corresponding adjunction
$$ \xymatrix{
\F(A) \ar^{f_!}[r]<1ex>  & \F(B) \ar^{f^*}_{\upvdash}[l]<1ex>  \\
}$$
is a Quillen adjunction: this can be seen by applying Lemma~\ref{l:car-lem2} to squares of the form
$$
\xymatrix{
X \ar[r]\ar_{\vphi}[d] & Y \ar^{f_!\vphi}[d] \\
X' \ar[r] & Y' \\
}$$
where $\vphi: X \lrar X'$ is a (trivial) cofibration covering the identity $A \lrar A$ and the horizontal maps are $\pi$-coCartesian lifts of $f: A \lrar B$. We can hence consider $\F$ as a functor $\M \lrar \ModCat$. According to Corollary~\ref{c:necessary} the functor $\F$ is proper. It is hence left to show that $\F$ is relative.

Let $f: A \lrar B$ be a weak equivalence in $\M$. Let $X \in \N$ be a cofibrant object lying over $A$ and $Y \in \N$ a fibrant object lying over $B$. Let $\psi:X \lrar f_!X$ be a coCartesian lift of $f$ starting at $X$ and let $\eta:f^*Y \lrar Y$ be a Cartesian lift of $f$ ending at $Y$. Then any map $\vphi: X \lrar Y$ lying over $f$ determines both a map $\phi: f_!X \lrar Y$ in $\F(B)$ by factoring along a $\psi$ and a map $\phi^{\ad}: X \lrar f^*Y$ by factoring along $\eta$, where $\phi$ and $\phi^{\ad}$ are adjoints with respect to $f_! \dashv f^*$. One then obtains a commutative diagram of the form
$$ \xymatrix{
X \ar^{\psi}[r]\ar^{\vphi}[dr]\ar_{\phi^{\ad}}[d] & f_!X \ar^{\phi}[d] \\
f^*Y \ar_{\eta}[r] & Y \\
}$$
According to property $(4)$ of Definition~\ref{d:model-fib} we know that $\psi$ and $\eta$ are weak equivalences. From the relative $2$-out-of-$3$ property (condition (2) of Definition~\ref{d:relative-model}) we deduce that
$$ \phi^{\ad} \in \W_{\N} \Leftrightarrow \vphi \in \W_{\N} \Leftrightarrow \phi \in \W_{\N} $$
and hence $f_! \dashv f^*$ is a Quillen equivalence. This proves claim $(1)$ above.

Let us now prove assertion $(2)$. We need to show that for each $A \in \M$ the induced adjunction
$$ \xymatrix{
\F(A) \ar^{\sig_A}[r]<1ex>  & \G(A) \ar^{\tau_A}_{\upvdash}[l]<1ex>  \\
}$$
is a Quillen adjunction. But this follows directly from the fact that for each $A \in \M$, the model categories $\F(A),\G(A)$ can be identified with the fibers $\int_\M\F\times_\M\{A\}$ and $\int_\M\F\times_\M \{B\}$ respectively with their inherited model structure and that $\sig_A \dashv \tau_A$ can be identified with the adjunction induced by $\sig \dashv \tau$.
\end{proof}

\section{Examples}\label{s:examples}
In this section we shall give several applications to Theorem \ref{qa}.

\subsection{(co)Slice categories}\label{ss:slice}
Let $\M$ is a model category. For every object $X\in \M$, we can consider the slice category $\M_{/X}$ of objects over $X$. This category can be endowed with a model structure (see~\cite{Hir}) in which a map $f: A \lrar B$ over $X$ is a weak equivalence (resp. fibration, cofibration) if and only if $f$ is a weak equivalence (resp. fibration, cofibration) in $\M$.  Every map $\vphi: X \lrar Y$ induces a Quillen adjunction
$$ \xymatrix{\M_{/X}\ar[r]<1ex>^{\vphi_{!}} & \M_{/Y}\ar[l]<1ex>_(0.5){\upvdash}^{\vphi^*}}$$
where $\vphi_{!}$ is given by composition with $\vphi$ and $\vphi^*$ is given by the fiber product.
\begin{lem}\label{slice proper}
For any model cateogry $\M$, the functor $\M_{/(-)}: \M \lrar \ModCat$ is proper.
\end{lem}
\begin{proof}
First observe $\vphi_{!}$ sends weak equivalences to weak equivalences for any $\vphi$.
Now assume that $\vphi:X\overset{\sim}{\thrar} Y$ is a trivial fibration and let $A\overset{\sim}{\lrar} B$ be a weak equivalence in $\M_{/Y}$. Consider the square
$$
\xymatrix{
A\times_Y X \ar[d] \ar[r] & B\times_Y X\ar[d]\\
A\ar[r] & B.
}$$
The maps $A\times_Y X\rar A$ and $B\times_Y X \lrar B$ are trivial fibrations as they were obtained by pulling back a trivial fibration. It follows by two-out-of-three that $A\times_Y X \lrar B\times_Y X$ is a weak equivalence.
\end{proof}

When $M$ is \textbf{right proper} the dependence of $\M_{/X}$ on $X$ is homotopy invariant, i.e., when $\vphi: X \lrar Y$ is a weak equivalence, the adjunction $\vphi_{!} \dashv \vphi^*$ is a Quillen equivalence (~\cite{Hir}). In other words, in this case we obtain a \textbf{relative functor}
$$ \M_{/(-)}: \M \lrar \ModCat $$
The category $\displaystyle\int_{\M} \M_{/(-)}$ is isomorphic to the arrow category $\M^{[1]}$ and Theorem~\ref{model structure} ensures that we get a model structure. Under this identification, this is precisely the injective model structure. In particular, we obtain the following:
\begin{cor}
Let $\M$ be a right proper model category. Then the projection
$$ \M^{[1]} \lrar \M $$
is a model fibration, where $\M^{[1]}$ is endowed with the injective model structure.
\end{cor}

\begin{rem}
Dually, for each $X\in \M$ one can consider the coslice category $\M_{X / }$. If $\M$ is left proper, the previous considerations dualize to show that $$ \M_{(-)/}: \M \lrar \ModCat $$ is a proper relative functor. In this case the model category of Theorem~\ref{model structure} is the projective model structure on the arrow category $\M^{[1]}$.
\end{rem}

\subsection{Group actions}\label{ss:group-act}

In this example we will show how to use Theorem~\ref{model structure} in order to obtain a (coarse) \textbf{global equivariant homotopy theory} for group actions. For such a theory to be widely applicable one would like to be able to work in a setting of coherent group actions. This will be undertaken thoroughly in a subsequent paper~\cite{HP} using Segal group actions as developed in~\cite{Pra}. In this subsection we will content with presenting a strict and a weak model for group actions and using Theorem~\ref{qa} to relate the two.

Throughout this section the word \textbf{space} will always mean a simplicial set. Let $\sGr$ be the category of simplicial groups. This category admits a model structure which is transferred from the Kan-Quillen model structure on spaces via the adjunction
$$ \xymatrix@=13pt{
\SS \ar[rr]<1ex>^(0.5){F} && \sGr \ar[ll]<1ex>_(0.5){\upvdash}^(0.5){U}
}$$
where $U$ is the forgetful functor and $F$ is the free group functor. In particular, a map of simplicial groups $f: G \lrar H$ is a weak equivalence (resp. fibration) if and only if the map $U(f): U(G) \lrar U(H)$ is a weak equivalence (resp. fibration). In addition, in this case if $f: G \lrar H$ is a cofibration then $U(f): U(G) \lrar U(H)$ is a cofibration as well.

For every simplicial group $G$ one can consider the category $\SS_G$ of spaces endowed with an action of $G$. This category can be identified with the simplicial functor category $\SS^{\B G}$ where $\B G$ is the simplicial groupoid with one object having $G$ as its automorphism group. As such one can consider $\SS^{\B G}$ with the \textbf{projective model structure}, also called the Borel model structure. In this model structure a map of $G$-spaces is a weak equivalence (resp. fibrations) if and only if it is such as a map of spaces. In addition, a $G$-space $X$ is cofibrant if and only if the action of $G$ on $X$ is \textbf{free} in each simplicial degree (see \cite[Proposition $2.2$ $(ii)$]{DDK}).

Now let $f: G \lrar H$ be a map of simplicial groups. Then we have a Quillen adjunction
$$ \xymatrix@=13pt{
\SS^{\B G} \ar[rr]<1ex>^(0.5){f_{!}} && \SS^{\B H} \ar[ll]<1ex>_(0.5){\upvdash}^(0.5){f^*}
}$$
where $f_{!}(X) = H \times_G X$ is the quotient of $H \times X$ by the action of $G$ given by $g(h,x) = (hg^{-1},gx)$.
and $f^*(X) = \res^H_G(X)$ is the restriction functor.

We then have the following.
\begin{pro}\label{p:G-spaces}
The functor $\U: \sGr \lrar \ModCat$ given by $\U(G) = \SS^{\B G}$ is proper and relative.
\end{pro}
\begin{proof}
We first prove that $\U$ is relative. Let $f: G \lrar H$ be a weak equivalence of simplicial groups. Since $f^*$ preserves and detects weak equivalences it will be enough to show that for each cofibrant $G$-space $X$ the unit map
$$ X \lrar f^*f_{!}(X) $$
is a weak equivalence. For this it will be enough to prove that for every cofibrant $G$-space $X$ the map
$$ X \lrar H \times_G X $$
is a weak equivalence of spaces. Since $X$ is cofibrant the action of $G$ on $X$ is free and hence the action of $G$ on $H \times X$ given by $g(h,x) = (hg^{-1},gx)$ is free as well, so that the quotient $H \times_G X = (H \times X)/G$ coincides with the homotopy quotient. Since the map $f: G \lrar H$ is a weak equivalence we get that the map $G \times X \lrar H \times X$ is a weak equivalence and so the induced map
$$ X = G \times_G X \lrar H \times_G X $$
is a weak equivalence as desired.

We shall now prove that $\U$ is proper. Since the restriction functors always preserve weak equivalences it will be enough to handle the left Quillen functors. Let $f: G \lrar H$ be a trivial cofibration of simplicial groups. Then $U(f): U(G) \lrar U(H)$ is a cofibration and so the action of $G$ on $H$ given by $g(h) = hf(g)^{-1}$ is free. This, in turn, implies that the action of $G$ on $H \times X$ is free for every $G$-space $X$ and so the quotient $H \times_G X = (H \times X)/G$ coincides with the homotopy quotient. This means that the functor $f_{!}(X) = H \times_G X$ preserves weak equivalences as desired.

\end{proof}

\begin{cor}
There exists a model structure on $\int_{G \in \sGr} \SS^{G}$ such that the projection
$$ \int_{G \in \sGr} \SS^{G} \lrar \sGr $$
is a model fibration.
\end{cor}

Let us now consider a weak model for group actions. We will say that a space $X \in \SS$ is \textbf{reduced} if $X_0 = \{*\}$ and will denote by $\SS_{0}$ the category of reduced spaces. According to Proposition $\mathrm{VI}.6.2$ of~\cite{GJ} there exists a model structure on $\SS_0$ in which the weak equivalences and cofibrations are those of the underlying spaces.

The full inclusion $\iota: \SS_0 \lrar \SS$ then becomes a left Quillen functor which preserves weak equivalences. We have the following observation:
\begin{pro}
The functor $\V:\SS_0 \lrar \ModCat$ defined by $\V(X) = \SS_{/\iota(X)}$  is proper and relative.
\end{pro}
\begin{proof}
Lemma $\mathrm{V}.6.6$ of~\cite{GJ} implies that $\iota$ preserves trivial fibrations. Since $\SS$ is right proper, the desired result now follows from the discussion in \S\ref{ss:slice}.
\end{proof}

We now wish to compare the functor $\V$ with the functor $\U$ discussed in Proposition~\ref{p:G-spaces}. 
For this we will consider the Quillen equivalence (see~\cite[$\mathrm{V}.6.3$]{GJ})
$$ \xymatrix@=13pt{\SS_0\ar[rr]<1ex>^(0.5){\GG} &&  \sGr \ar[ll]<1ex>_(0.5){\upvdash}^(0.5){\ovl{W}}} $$
where $\GG$ is the Kan loop group functor. We wish to present a \textbf{right Quillen equivalence} (see Definition~\ref{d:left-right-equiv}) from $\V$ to $\U$. For this we need to describe a compatible family of Quillen adjunctions
$$\xymatrix{
\Theta^L_G:\SS_{/\iota\ovl{W}(G)} \ar[r]<1.1ex> & \SS^{\B G}:\Theta^R_G\ar[l]<1.1ex>_(0.45){\upvdash}}$$
indexed by $G \in \sGr$, which are equivalences for fibrant $G$ (i.e. for all $G$). Such a compatible family is provided by the work of~\cite{DDK}. More explicitly, for every $G$-space $X$ one defines
$$ \Theta^R_G(X) = (W(G) \times X)/G $$
with its natural map $(W(G) \times X)/G \lrar \ovl{W}(G)$ given by the projection on the first coordinate (where $W: \sGr \lrar \SS$ is defined as in~\cite[V.4]{GJ}). To see that this family of Quillen equivalences is indeed compatible (i.e., constitutes a pseudo-natural transformation of adjunctions), one has to verify that for each map $G \lrar H$ of simplicial groups and every $H$-space $X$ the natural commutative diagram of simplicial sets
$$ \xymatrix{
(W(G) \times X)/G \ar[r]\ar[d] & (W(H) \times X)/H \ar[d] \\
\ovl{W}(G) \ar[r] & \ovl{W}(H) \\
}$$
is Cartesian. This can be verified directly using the fact that the action of $G$ on $W(G)$ (and the action of $H$ on $W(H)$) is free in each simplicial degree.

In light of Theorem~\ref{qa} and the above we obtain the following
\begin{cor}\label{c:equivalence-1}
There exists a Quillen equivalence
$$
\xymatrix{\Phi^L: \displaystyle\mathop{\int}_{X \in \SS_0}\SS_{/\iota(X)} \ar[r]<2ex> & \displaystyle\mathop{\int}_{G \in \sGr}\SS^{\B G}:\Phi^R\ar[l]<0.7ex>_(0.5){\upvdash}.}
$$
\end{cor}

\subsection{Modules over associative algebras}\label{ss:alg}

Let $\M$ be a symmetric monoidal model category and let $\Alg(\M)$ be the category of associative algebra objects in $\M$ (i.e. objects equipped with a unital and associative multiplication). We have an adjunction
$$ \xymatrix@=13pt{\M\ar[rr]<1ex>^(0.4){T} && \Alg(\M)\ar[ll]<1ex>_(0.6){\upvdash}^(0.6){U}}$$
where $U$ is the forgetful functor and $T$ is the free algebra functor. For each algebra object $R \in \Alg(\M)$ one can consider the categories $\LMod(R)$ and $\RMod(R)$ of \textbf{left $R$-modules} and \textbf{right $R$-modules} respectively. We have similar adjunctions
$$ \xymatrix@=13pt{\M\ar[rr]<1ex>^(0.4){R \otimes (-)} &&  \LMod(R)\ar[ll]<1ex>_(0.6){\upvdash}^(0.6){U}}$$
and
$$ \xymatrix@=13pt{\M\ar[rr]<1ex>^(0.4){(-) \otimes R} && \RMod(R)\ar[ll]<1ex>_(0.6){\upvdash}^(0.6){U}}$$
where $U$ always denotes the forgetful functor.

\begin{rem}\label{r:unit}
The unit $\II \in \M$ of the monoidal structure carries a canonical algebra structure. With respect to this structure the adjunctions above give isomorphisms $\LMod(\II) \cong \M$ and $\RMod(\II) \cong \M$.
\end{rem}

In~\cite{SS}, Schwede and Shipley establish the existence of model structures on the categories $\Alg(\M), \LMod(R)$ and $\RMod(R)$ under suitable assumptions. The most significant of those assumptions is known as the \textbf{monoid axiom}. In order to phrase this axiom it will be useful to introduce the following terminology.

\begin{define}\label{d:weakly saturated}
Let $\M$ be a model category. A class of morphisms $\U\subseteq \M$ is said to be \textbf{weakly saturated} if it is closed under pushouts, transfinite compositions and retracts. 
\end{define}

Note that an intersection of weakly saturated classes of morphisms is again weakly saturated.

\begin{define}\label{d:U}
Let $\M$ be a symmetric monoidal model category and let $R \in \Alg(\M)$ be an algebra object. We will denote by $\U_R \subseteq \RMod(R)$ the smallest weakly saturated class of morphisms containing all the morphisms of the form
$$ f \otimes M: K \otimes M \lrar L \otimes M $$
where $f: K \lrar L$ is a trivial cofibration in $\M$ and $M \in \RMod(R)$ is a right $R$-module.
\end{define}

We are now ready to formulate the monoid axiom.
\begin{define}[\cite{SS}, Definition 3.3] 
Let $\M$ be a symmetric monoidal model category with unit $\II \in \M$. We will say that $\M$ satisfies the \textbf{monoid exiom} if every morphism in $\U_\II$ is a weak equivalence (where $\II$ is carries its canonical algebra strucute, see Remark~\ref{r:unit}).
\end{define}

The following theorem is one of the main results of~\cite{SS}:
\begin{thm}[\cite{SS}, Theorem 4.1]\label{t:ss}
Let $\M$ be a combinatorial symmetric monoidal model category satisfying the monoid axiom. Then:
\begin{enumerate}
\item
The category $\Alg(\M)$ of associative algebra objects in $\M$ can be endowed with a combinatorial model structure in which a map $f: R \lrar S$ of algebra objects is a weak equivalence (resp. fibration) if and only if $U(f)$ is a weak equivalence (resp. fibration) in $\M$.
\item
Let $R \in \Alg(\M)$ be an algebra object. Then the category $\LMod(R)$ of left $R$-modules can be endowed with a combinatorial model structure in which a map $f: M \lrar N$ of $R$-modules is a weak equivalence (resp. fibration) if and only if $U(f)$ is a weak equivalence (resp. fibration) in $\M$. The analogous statement for $\RMod(R)$ holds as well.

\end{enumerate}
\end{thm}
Now any map of algebras $f: R \lrar S$ induces Quillen a adjunction
$$ \xymatrix@=13pt{
\LMod(R)\ar[rr]<1ex>^(0.5){S \otimes_R (-)} && \LMod(S)\ar[ll]<1ex>_(0.5){\upvdash}^(0.5){\res^S_R}
}$$
where $S \otimes_R M$ is given by the coequlizer of
$$ \xymatrix{
S \otimes R \otimes M \ar[r]<0.5ex>\ar[r]<-0.5ex> & S \otimes M \\
}$$
and $\res^S_R$ is the functor which restricts the action from $S$ to $R$. Similarly, there is an analogous Quillen adjunction for right modules. By associating the above adjunction to any such $f: R \lrar S$ one may endow the associations $R \mapsto \LMod(R)$ and $R \mapsto \RMod(R)$ with structures of \textbf{functors} $\Alg(\M) \lrar \ModCat$. 

\begin{rem}
Note that the restriction functors $\res^S_R$ always preserve and reflect weak equivalences.
\end{rem}

In~\cite{SS} the authors introduce an additional assumption on $\M$ which is important for our purposes. We will formulate their assumption using the following notation.
\begin{define}
Let $\M$ be a monoidal model category and let $R$ be an algebra object. We will say that a left $R$-module is \textbf{flat} if the operation $(-) \otimes_R N$ takes weak equivalences of right $R$-modules to weak equivalences in $\M$. Similarly, we will say that a right $R$-module is flat if the operation $M \otimes_R (-)$ takes weak equivalences of left $R$-modules to weak equivalences in $\M$.
\end{define}

\begin{define}\label{d:SS-assume}
Let $\M$ be a combinatorial symmetric monoidal model category. We will say that $\M$ satisfies the \textbf{Schwede-Shipley assumptions} if:
\begin{enumerate}
\item
$\M$ satisfies the monoid axiom.
\item
Every cofibrant (left or right) $R$-module in $\M$ is flat.
\end{enumerate}
\end{define}

\begin{example}\label{e:many}
All the model categories described in \S 5 of~\cite{SS} satisfy Schwede-Shipley assumptions, except for the category of $S$-modules, which is not combinatorial. This includes, for example, the models of symmetric and orthogonal spectra (with the stable model structure).
\end{example}

We then have the following reformulation of a theorem of~\cite{SS}:
\begin{thm}[\cite{SS}, Theorem 4.3]
Let $\M$ be a combinatorial symmetric monoidal model category which satisfies the Schwede-Shipley assumptions. Then the functors $R \mapsto \LMod(R)$ and $R \mapsto \RMod(R)$ are \textbf{relative}.
\end{thm}

Our goal now is to prove that if $\M$ satisfies the Schwede-Shipley assumptions then the functors $R \mapsto \LMod(R)$ and $R \mapsto \RMod(R)$ are also \textbf{proper}. For this we will need the following notion:
\begin{define}\label{d:flat}
Let $R \in \Alg(\M)$ be an algebra object. We will say that a map $M \lrar N$ of right (resp. left) $R$-modules is a \textbf{flat equivalence} if for every left (resp. right) $R$-module $O$ the induced map
$$ M \otimes_R O \lrar N \otimes_R O \;\;\;\; (\text{resp.}\; O \otimes_R M \lrar O \otimes_R N) $$
is a weak equivalence in $\M$.
\end{define}

\begin{rem}
Taking $O = R$ in Definition~\ref{d:flat} we see that every flat equivalence is a weak equivalence.
\end{rem}

\begin{notn}
Given a map of algebras $R \lrar S$, the object $S$ inherits canonical structures of
both a right and a left $R$-module. We will denote the resulting right $R$-module by $S_{/R}$ and the resulting left $R$-module by $S_{R/}$.
\end{notn}

\begin{thm}\label{t:rmod}
Let $\M$ be a combinatorial symmetric monoidal model category which satisfies the Schwede-Shipley assumptions. Then for any trivial cofibration $R \lrar S$ in $\Alg(\M)$ the induced maps $h_{/R}: R_{/R} \lrar S_{/R}$ and $h_{R/}: R_{R/} \lrar S_{R/}$ are flat equivalences (of right and left $R$-modules respectively).
\end{thm}

Before proceeding to prove Theorem~\ref{t:rmod} let us explain how it implies that the functors $\LMod(-)$ and $\RMod(-)$ are proper.

\begin{thm}\label{t:rel-prop}
Let $\M$ be a combinatorial symmetric monoidal model category model category which satisfies the Schwede-Shipley assumptions (see Definition~\ref{d:SS-assume}). Then the functors $\LMod(-),\RMod(-): \M \lrar \ModCat$ are proper.
\end{thm}
\begin{proof}
We will prove the claim for $\LMod(-)$. The proof for $\RMod(-)$ is completely analogous and will be omitted. 
Since all restriction functors preserve weak equivalences we can focus attention on the left Quillen functors. Let $f: R \lrar S$ be a trivial cofibration of algebras. According to Theorem~\ref{t:rmod} the map $R_{/R} \lrar S_{/R}$ is flat equivalence of right $R$-modules. This implies that for every left $R$-module $M$ the unit map
$$ M \lrar \res^S_R\left(S \otimes_R M\right) $$
is a weak equivalence. Since $\res^S_R$ reflects weak equivalences this implies that $S_{/R} \otimes_R (-)$ preserves weak equivalences.
\end{proof}

\begin{cor}
Let $\M$ be a combinatorial symmetric monoidal model category which satisfies the Schwede-Shipley assumptions. Then there exists a model structure on
$$ \int_{R \in \Alg(\M)} \LMod(R) $$
such that the projection
$$ \int_{R \in \Alg(\M)} \LMod(R) \lrar \Alg(\M) $$
is a model fibration. The analogous statement for right modules holds as well.
\end{cor}

\begin{rem}
The category $\int_{R \in \M} \LMod(R)$ can also be described as the category of $\Alg_{\O}(\M)$ of algebras over a certain coloured operad $\O$ (see~\cite[1.5]{BM}). Furthermore, the integral model structure on $\Alg_{\O}(\M)$ is the one transferred from $\M$ along the forgetful functor $\Alg_{\O}(\M) \lrar \M$. To see this, we first observe that the fibrations are defined in the same way. To see that the weak equivalences coincide, let us revoke Observation~\ref{o:sym} which says, in this case, that a map $(f,\vphi):(R,M) \lrar (S,N)$ in $\int_{R \in \M} \LMod(R)$ is a weak equivalence if and only if $f: R \lrar S$ is a weak equivalence in $\Alg(\M)$ and the composite $M \x{\vphi^{\ad}}{\lrar} \res^S_R(N) \lrar \res^S_R(N^{\fib})$ is a weak equivalence in $\LMod(R)$. Since $\res^S_R$ preserves weak equivalences this is the same as saying that $\vphi^{\ad}$ itself is a weak equivalence. But this is exactly the definition of weak equivalences in the transferred model structure.

To the best of the authors' knowledge, in the case where $\M$ is a model for spectra, the transferred model structure was only known to exist for the positive model structure on symmetric spectra (see~\cite{EM}) and the positive model structure on orthogonal spectra (see~\cite{Kro}). Theorem~\ref{t:rel-prop}, on the other hand, provides a criterion for the existence of the transferred model structure on $\Alg_{\O}(\M)$ which applies to most models of spectra (as those appearing in Example~\ref{e:many}).
\end{rem}

The rest of this subsection is devoted to the the proof of Theorem~\ref{t:rmod}. We will focus on proving the claim for $h_{/R}$. The proof for $h_{R/}$ is completely analogous and will be omitted.
%

The proof of Theorem~\ref{t:rmod} will be achieved through the following lemmas:
\begin{lem}\label{l:flat-equiv}
Let $\M$ be a symmetric monoidal model category satisfying the monoid axiom and let $R$ be an algebra object in $M$. Then every morphism in $\U_R$ (see Definition~\ref{d:U}) is a flat equivalence.
\end{lem}
\begin{proof}
Let $O \in \LMod(R)$ be a left $R$-module. Since the functor $(-) \otimes_R O: \RMod(R) \lrar \M$ preserves pushouts, transfinite compositions and retracts it follows that the class of morphisms
$\U_R \otimes_R O$ is contained in the class $U_\II$ (where $\II \in \M$ is the unit object). According the monoid axiom every morphism in $U_\II$ is a weak equivalence. By definition it follows that every morphism in $\U_R$ is a flat equivalence.
\end{proof}

\begin{lem}\label{l:technical}
Let
$$ \xymatrix{
R \ar^{h}[rr] && S \\
& A \ar[ur]\ar[ul] & \\
}$$
be a commutative diagram in $\Alg(\M)$ such that $h$ is a trivial cofibration of algebras. Then the induced map
$$ h_{/A}:R_{/A} \lrar S_{/A} $$
belongs to $U_A$.
\end{lem}

\begin{proof}
Since the operation $h \mapsto h_{/A}$ preserves retracts and transfinite compositions it will be enough to prove the claim for $h$ of the form
$$ \xymatrix{
 T(K) \ar^{T(f)}[rr]\ar^{g}[d] && T(L) \ar[d] \\
 R \ar^{h}[rr] && S \\
 & A \ar[ur]\ar^{\iota}[ul] & \\
}$$
where $f: K \lrar L$ is a trivial cofibration in $\M$ and the top square is a pushout square in $\Alg(\M)$.

We shall adapt the main construction in the proof of Lemma $6.2$ in~\cite{SS}, namely the tower
\begin{equation}\label{e:tower}
R = P_0 \x{\rho_0}{\lrar} P_1 \x{\rho_1}{\lrar} ... \lrar P_n \lrar ...
\end{equation}
whose colimit is the underlying object $U(S)$, to a tower of right $A$-modules. As in~\cite{SS}, we first define the auxiliary objects $Q_n$. Let $\P(\{1,...,n\})$ denote the cube category, i.e., the poset of subsets of $\{1,...,n\}$. Consider the functor
$$ W_n: \P(\{1,...,n\}) \lrar \RMod(A) $$
given by
$$ W_n(\Sig) = R \otimes C_1 \otimes R \otimes C_2 \otimes ... \otimes C_n \otimes R_{/A} $$
where
$$ C_i = \begin{cases} K \;\;\text{  if } i \notin \Sig \\ L \;\;\text{  if } i \in \Sig \end{cases} $$
In other words, the $A$-module structure on $W_n(\Sig)$ is given by multiplication with the right-most $R$-factor. We then set $Q_n$ to be the colimit of $W_n$ restricted to the punctured cube, i.e.,
$$ Q_n = \mathop{\colim}_{\Sig \subsetneq \{1,...,n\}} W_n(\Sig) \in \RMod(A) $$
We shall now define inductively an object $P^{\iota}_n \in \RMod(A)$ (such that $U(P^{\iota}_n)$ is $P_n$ in~\cite{SS}) together with a map
$$ Q_{n+1} \lrar P^{\iota}_n $$
For $n=0$ we set $P^{\iota}_0 = R_{/A}$ and the map
$$ Q_1 = R \otimes K \otimes R_{/A} \lrar R_{/A} $$
is given as the composite
$$
\xymatrix{
R \otimes K \otimes R_{/A} \ar[rr]^{\Id \otimes g^{\ad} \otimes \Id} &&
R \otimes R \otimes R_{/A} \ar[r] & R_{/A}  \\
}$$
where $g^{ad}: K \lrar R$ is the adjoint of $g$ in $\M$ and the last map multiplies all the $R$ factors (using the algebra structure of $R$). Now assume that $P^{\iota}_{n-1} \in \RMod(A)$ was given together with a map of right $R$-modules
$$ Q_n \lrar P^{\iota}_{n-1} $$
We define $P^{\iota}_n$ to be the pushout in $\RMod(A)$ of
$$ \xymatrix{
Q_n \ar[r]\ar[d] & W_n(\{1,...,n\}) \ar[d] \\
P^{\iota}_{n-1} \ar[r] & P^{\iota}_n \\
}$$
To define the desired map $Q_{n+1} \lrar P^{\iota}_n$ it will suffice to give a compatible collection of maps
$$ W_{n+1}(S) \lrar P^{\iota}_n $$
for $\Sig \subsetneq \{1,...,n+1\}$. Each of the factors of $W_{n+1}(\Sig)$ which is equal to $K$ is first mapped into $R$ via $g^{\ad}$. The adjacent factors of $R$ are then multiplied, yielding a map
$$ W_{n+1}(\Sig) \lrar  W_{|\Sig|}(\{1,...,|\Sig|\}) = R \otimes L \otimes R \otimes ... \otimes L \otimes R_{/A}  $$

The right-hand side then maps further to $P^{\iota}_{|\Sig|}$, and hence to $P^{\iota}_n$ since $|\Sig| \leq n$.

Now let $P^{\iota} = \colim_n P^{\iota}_n$ be the colimit of the tower in $\RMod(A)$. Since we have a compatible family of maps $\vphi_n: A_{/A} \lrar P_n$ we have an induced map $\vphi:A_{/A} \lrar P^{\iota}$. We claim that $P^{\iota}$ is isomorphic to $S_{/A}$ as a right $A$-module. To see this, recall that the forgetful functor
$$ U: \RMod(A) \lrar \M $$
commutes with colimits and so
$$ U(P^{\iota}) \cong \colim_n U(P^{\iota}_n) $$
According to the proof of Lemma $6.2$ in~\cite{SS} the colimit on the right hand side can be identified with the underlying object $U(S)$ of $S$ and in particular carries an algebra structure. Hence we conclude that
$$ U(P^{\iota}) \cong U(S) $$
In particular, the underlying object of $P^{\iota}$ does not depend on $\iota$. It is hence left to show that the $A$-module structure on $P^{\iota}$ coincides with the $A$-module structure of $S_{/A}$.

We now observe that the $A$-module structure on $P^{\iota}$ depends functorially on $\iota$, in the sense that if $f:A' \lrar A$ is a map of algebras then $P^{\iota \circ f}$ is obtained from $P^{\iota}$ by restricting the module structure along $f$. Since the $A$-module structure of $S_{/A}$ is functorial in $A$ in the same manner, we see that in order to prove that the $A$-module structures of $P^{\iota}$ and $S_{/A}$ coincide one can assume that $\iota$ is the identity $A \lrar A$.

We now need to prove that the $A$-module structure on $P^{\Id}$ factors through the algebra structure of $S$ in the sense that the following square
\begin{equation}\label{e:module}
\xymatrix{
U(P^{\Id}) \otimes U(A) \ar[r]\ar_{\Id \otimes U(\vphi)}[d] & U(P^{\Id}) \ar@{=}[d] \\
U(P^{\Id}) \otimes U(P^{\Id}) \ar[r] & U(P^{\Id}) \\
}
\end{equation}
commutes. Recall that the algebra structure on $\colim_n U(P^{\Id}_n)$ was given by a compatible family of graded products
$$ U(P^{\Id}_n) \otimes U(P^{\Id}_m) \lrar U(P^{\Id}_{n+m}) $$
The desired commutativity of~\ref{e:module} will now follow once we verify that the diagrams
$$ \xymatrix{
U(P^{\Id}_n) \otimes U(A) \ar[r]\ar_{\Id \otimes U(\vphi_m)}[d] & U(P^{\Id}_n) \ar[d] \\
U(P^{\Id}_n) \otimes U(P^{\Id}_m) \ar[r] & U(P^{\Id}_{n+m}) \\
}$$
are commutative. Since the graded products are compatible with each other it will suffice to prove for $m=0$, i.e., to show that the diagram
$$ \xymatrix{
U(P^{\Id}_n) \otimes U(A) \ar[r]\ar_{\Id \otimes U(\iota)}[d] & U(P^{\Id}_n) \ar[d] \\
U(P^{\Id}_n) \otimes U(A) \ar[r] & U(P^{\Id}_n) \\
}$$
is commutative. This in turn is true since both vertical morphisms are the identity. This concludes the promotion of the tower in~\cite{SS} to a tower of $A$-modules converging to the $A$-module $S_{/A}$.

Now recall that we wish to show that the map
$$ R_{/A} \lrar S_{/A} $$
is in $U_A$. In light of the above it will suffice to show that each
$$ \rho_n: P^{\iota}_n \lrar P^{\iota}_{n+1} $$
is in $U_A$. By the definition of $\rho_n$ as a pushout we see that it will be enough to show that the map
$$ Q_n \lrar W_n(\{1,...,n\}) $$
is in $U_A$. But this is due to the fact that we can write this map as
$$ \ovl{Q}_n \otimes M \lrar L^{\otimes n} \otimes M $$
where $M = R \otimes ... \otimes R \otimes R_{/A} \in \RMod(R)$ is a right $R$-module and $\ovl{Q}_n \lrar L^{\otimes n}$ is the iterated pushout-product of the trivial cofibration $f: K \lrar L$ with itself, hence a trivial cofibration.

\begin{proof}[Proof of Theorem~\ref{t:rmod}]
Let $R \lrar S$ be a trivial cofibration of algebras. Lemma~\ref{l:technical} applied to the case $R = A$ implies that the map $h_{/R}: R_{/R} \lrar S_{/R}$ belongs to $\U_R$. By Lemma~\ref{l:flat-equiv} the map $h_{/R}$ is a flat equivalence as desired.
\end{proof}

\end{proof}

\subsection{Modules over commutative algebras }\label{ss:calg}

Let $\M$ be a symmetric monoidal model category and let $\CAlg(\M)$ be the category of commutative algebra objects in $\M$ (i.e. objects equipped with a unital, associative and commutative multiplication). We have an adjunction
$$ \xymatrix@=13pt{\M\ar[rr]<1ex>^(0.4){F} && \CAlg(\M)\ar[ll]<1ex>_(0.6){\upvdash}^(0.6){U}}$$
where $U$ is the forgetful functor and $F$ is the free commutative algebra functor. When $\M$ is also a symmetric monoidal \textbf{model category}, it is natural to ask whether the model structure on $\M$ can be transferred to $\CAlg(\M)$ along the adjunction $F \dashv U$. This case is known to be more subtle than the analogous case of associative algebras. For model categories which model spectra, the answer is known to be negative for several prominent examples. In~\cite{Shi} a model structure on symmetric spectra was constructed, now known as the \textbf{positive flat} stable model structure, which could indeed be transferred along the adjunction $F \dashv U$. This was later generalized by Lurie (see~\cite{Lur11}) to model categories which are \textbf{free powered}. Lurie also showed that in this case the resulting model category $\CAlg(\M)$ models the $\infty$-category of $E_\infty$-algebra objects in $\M$. If one is only interested in the existence of the transferred model structure, weaker conditions were established by White (\cite{Wh}). However, one should be mindful that under the assumptions of~\cite{Wh} the comparison between the resulting model structure and its $\infty$-categorical analogue may fail in general, see Example~\ref{e:space} below.

The following definition is~\cite[Definition 4.4.4.2]{Lur11}.
\begin{define}\label{powerdef}
Let $\M$ be a symmetric monoidal model category. We will denote by
$$ \wedge^{n}(f): \Box^{n}(f) \lrar Y^{\otimes n} $$
the iterated pushout-product of $f$ with itself. Note that both $\Box^{n}(f)$ carry a natural action of the symmetric group $\Sig_n$. We will denote by
$$ \sigma^{n}(f): \Sym^{n}(Y; X) \lrar \Sym^{n}(Y) $$
the induced map on $\Sig_n$-coinvariants.
\end{define}

The main assumption used in~\cite{Wh} is the following:
\begin{define}
Let $\M$ be a symmetric monoidal model category. We will say that $\M$ satisfies the \textbf{commutative monoid axiom} if for every trivial cofibration $f: X \lrar Y$ in $\M$ and every $n > 0$ the corresponding map
$$ \sigma^{n}(f): \Sym^{n}(Y; X) \lrar \Sym^{n}(Y) $$
is a trivial cofibration.
\end{define}

The following theorem is essentially Theorem $3.2$ of~\cite{Wh}:
\begin{thm}\label{t:white}
Let $\M$ be a combinatorial symmetric monoidal model category which satisfies both the monoid axiom and the commutative monoid axiom. Then there exists a combinatorial model structure on $\CAlg(\M)$ such that a map $f: A \lrar B$ of commutative algebras is a weak equivalence (resp. fibration) if and only if $U(f)$ is a weak equivalence (resp. fibration) in $\M$.
\end{thm}

\begin{example}\label{e:comp}
Let $k$ be a field of characteristic $0$. The category $\Ch(k)$ of (unbounded) chain complexes over $k$ with the projective model structure satisfies the assumptions of Theorem~\ref{t:white}.
\end{example}

\begin{example}\label{e:positive}
Let $\M$ be the category of symmetric spectra endowed with the \textbf{positive flat stable} model structure (see~\cite{Shi}). Then $\M$ satisfies the assumptions of Theorem~\ref{t:white} (see~\cite[Theorem $5.7$]{Wh}).
\end{example}

\begin{example}\label{e:space}
The category $\Set_\Del$ of simplicial sets with the Kan Quillen model structure satisfies the assumptions of Theorem~\ref{t:white} with respect to the Cartesian product. However, unlike the two examples above, $\Set_\Del$ does not satisfy the stronger conditions appearing in~\cite{Lur11}. In particular, as is well known, the resulting model category $\CAlg(\Set_\Del)$ is \textbf{not} a model for the $\infty$-category of $E_\infty$-monoids in spaces. 
\end{example}

For each commutative algebra object $A \in \CAlg(\M)$ one can consider the category $\Mod(A)$ of \textbf{$A$-modules} (since $A$ is commutative the categories of left modules and right modules coincides so one can just talk of modules). We have a similar adjunction
$$ \xymatrix@=13pt{\M\ar[rr]<1ex>^(0.4){A \otimes (-)} &&  \Mod(A)\ar[ll]<1ex>_(0.6){\upvdash}^(0.6){U}} $$
Recall the following theorem which is essentially taken from~\cite{SS}:
\begin{thm}\label{t:ssc}
Let $\M$ be a combinatorial symmetric monoidal model category satisfying the monoid axiom. Let $A \in \CAlg(\M)$ be a commutative algebra object. Then the category $\Mod(A)$ of $A$-modules can be endowed with a combinatorial model structure in which a map $f: M \lrar N$ of $A$-modules is a weak equivalence (resp. fibration) if and only if $U(f)$ is a weak equivalence (resp. fibration) in $\M$.
\end{thm}
Any map of algebras $f: A \lrar B$ induces Quillen a adjunction
$$ \xymatrix@=13pt{
\Mod(A)\ar[rr]<1ex>^(0.5){B \otimes_A (-)} && \Mod(B)\ar[ll]<1ex>_(0.5){\upvdash}^(0.5){\res^B_A}
}$$
where $B \otimes_A M$ is given by the coequlizer of
$$ \xymatrix{
B \otimes A \otimes M \ar[r]<0.5ex>\ar[r]<-0.5ex> & B \otimes M \\
}$$
and $\res^B_A$ is the functor which restricts the action from $B$ to $A$. Our goal for the rest of this subsection is to show that the functor $A \mapsto \Mod(A)$ is proper is relative. Our strategy is similar to the case of associative algebras.

\begin{rem}
Given a map of commutative algebras $A \lrar B$ the object $B$ inherits canonical structure of
of an $A$-module, which we will denote by $B_{/A} \in \Mod(A)$.
\end{rem}

\begin{define}
Let $A$ be a commutative algebra object in $\M$. As in the case of associative algebras (see Definition~\ref{d:U}), we will denote by $U_A \subseteq \Mod(A)$ the smallest weakly saturated class of morphisms containing all the morphisms of the form
$$ f \otimes M: K \otimes M \lrar L \otimes M $$
where $f: K \lrar L$ is a trivial cofibration in $\M$ and $M \in \Mod(A)$ is an $A$-module.
\end{define}

The following lemma is the commutative analogue of Lemma~\ref{l:technical}, although its proof is much simpler. The commutative analogue of the Schwede-Shipley tower appeared in~\cite{Shi},~\cite{Lur11} and~\cite{Wh}.
\begin{lem}\label{l:technical-2}
Let $\M$ be a symmetric monoidal model category which satisfies the assumptions of~\ref{t:white}. Let
$$
\xymatrix{
A \ar^{h}[rr] &&  B \\
& C \ar[ur]\ar[ul] \\
}$$
be a commutative diagram in $\CAlg(\M)$ such that $h$ is a trivial cofibration of commutative algebras. Then the induced map
$$ h_{/C}: A_{/C} \lrar B_{/C} $$
is belongs to $U_A$.
\end{lem}
\begin{proof}
Since the operation $h \mapsto h_{/C}$ preserves retracts and transfinite compositions it will be enough to prove the claim for $h$ of the form
$$ \xymatrix{
 T(K) \ar^{T(f)}[rr]\ar^{g}[d] && T(L) \ar[d] \\
 A \ar^{h}[rr] && B \\
 & C \ar[ur]\ar^{\iota}[ul] & \\
}$$
where $f: K \lrar L$ is a trivial cofibration in $\M$ and the top square is a pushout square in $\CAlg(\M)$. As in the proof of Lemma~\ref{l:technical} the $C$-module $B_{/C}$ admits a filtration by $C$-modules
$$ A \simeq B_0 \lrar B_1 \lrar B_2 \lrar \ldots,$$
where $B \simeq \colim \{B_i\}$ and for each $n > 0$ there is a pushout diagram of $C$-modules
$$ \xymatrix{
A \otimes \Sym^{n}(L;K) \ar[rr]^{ A \otimes \sigma^{n}(i) } \ar[d] & & A \otimes \Sym^{n}(L) \ar[d] \\
B_{n-1} \ar[rr] & &  B_{n}.
}$$
Since $\M$ satisfies the commutative monoid axiom the maps $\sig^n(i)$ are trivial cofibrations in $\M$. By definition we then get that  $h_{/C}$ belongs to $U_C$.
\end{proof}

\begin{cor}\label{c:proper}
Let $\M$ be a symmetric monoidal model category which satisfies the assumptions Theorem~\ref{t:white}. Let
$$ h: A \lrar B $$
be a trivial cofibration of commutative algebras. Then the map $A_{/A} \lrar B_{/A}$ is a flat equivalence (see Definition~\ref{d:flat}).
\end{cor}
\begin{proof}
This follows directly from Lemma~\ref{l:technical-2} and Lemma~\ref{l:flat-equiv}.
\end{proof}

\begin{define}
We will say that a symmetric monoidal model category $\M$ is \textbf{commutatively flat} if it satisfies the following property: for every commutative algebra object $A$ and every cofibrant $A$-module $M$, the operation $(-) \otimes_A M$ takes weak equivalences of $A$-modules to weak equivalences in $\M$.
\end{define}

\begin{thm}\label{t:rel-prop-com}
Let $\M$ be a symmetric monoidal model category which satisfies the assumptions of Theorem~\ref{t:white}. Assume in addition that $\M$ is commutatively flat. Then the functor $\Mod(-): \M \lrar \ModCat$ is relative and proper.
\end{thm}
\begin{proof}
The fact that under these assumptions the functor $\Mod(-)$ is relative is proved in~\cite[Theorem 4.1]{Wh}. Let us now show that $\Mod(-)$ is proper. 

Since all restriction functors preserve weak equivalences we can focus attention on the left Quillen functors. Let $f: A \lrar B$ be a trivial cofibration of commutative algebras. According to Corollary~\ref{c:proper} the map $A_{/A} \lrar B_{/A}$ is flat equivalence. This implies that the unit map
$$ M \lrar \res^B_A\left(B \otimes_A M\right) $$
is a weak equivalence for every $M$. Since $\res^B_A$ reflects weak equivalences this implies that $B_{/A} \otimes_A (-)$ preserves weak equivalences.
\end{proof}

\begin{cor}\label{c:com}
Let $\M$ be a commutatively flat symmetric monoidal model category which satisfies the assumptions of Theorem~\ref{t:white}. Then there exist a model structure on
$$ \int_{A \in \CAlg(\M)} \Mod(A) $$
such that the projection
$$ \int_{A \in \CAlg(\M)} \Mod(A) \lrar \CAlg(\M) $$
is a model fibration.
\end{cor}

\begin{example}
Examples~\ref{e:comp},~\ref{e:positive} and~\ref{e:space} all satisfy the assumptions of Corollary~\ref{c:com}.
\end{example}


\begin{thebibliography}{Lur09}
\bibitem[Bar]{Bar} C.~Barwick, \emph{On left and right model categories and left and right Bousfield localizations}, Homology, Homotopy Appl. 12 no. 2, 245--320, (2010).

\bibitem[Ber]{Be} J.~Bergner, \emph{Homotopy limits of model categories and more general homotopy theories}, Bulletin of the London Mathematical Society 44.2, p.311--322, (2012).

\bibitem[BM]{BM} C.~Berger and I.~Moerdijk \emph{Resolutions of coloured operads and rectification of homotopy algebras}, Categories in algebra, geometry and mathematical physics, 31--58, Contemp. Math , 431, Amer. Math. Soc., Providence RI (2007).

\bibitem[DDK]{DDK} E.~Dror, W.~Dwyer and D.M.~Kan,
\emph{Equivariant maps which are self homotopy equivalences}, Proceedings of the American Mathematical Society, 80, No. 4, p. 670--672 (1980).

\bibitem[EM]{EM} A.~D.~Elmendorf and M.~A.~Mandell, \emph{Ring, modules and algebras in infinite loop space theory}, Adv. Math., 205(1), 163--228 (2006).


\bibitem[GHN]{GHN}
D.~Gepner, R.~Haugseng, T.~Nikolaus, \emph{Lax colimits and free fibrations in $\infty$-categories}, preprint arXiv:1501.02161 (2015).


\bibitem[GJ]{GJ} P.~Goerss and R.~Jardine, \emph{Simplicial Homotopy Theory}, Vol. 174. Springer (2009).

\bibitem[Gra]{Gra} J.~W.~Gray \emph{Formal category theory: adjointness for 2-categories}, Lecture Notes in Mathematics 391, Springer-Verlag (1974).

\bibitem[Gro]{SGA1} 
A. Grothendieck, \emph{Rev\^etements \'etales et groupe fondamental}, Institut des Hautes Etudes Scientifiques (1964).


\bibitem[GS]{GS} J.~P.~C.~Greenlees, B.~Shipley, \emph{Homotopy theory of modules over diagrams of rings}, preprint arXiv:1309.6997 (2013).



\bibitem[Hin]{Hin} V.~Hinich, \emph{Dwyer-Kan localization revisited}, preprint arXiv:1311.4128 (2013).

\bibitem[Hir]{Hir} P.~S.~ Hirschhorn, \emph{Model categories and their localizations}, Mathematical Surveys and Monographs, 99. American Mathematical Society, Providence, RI, xvi+457 pp. (2003).

\bibitem[HP]{HP} Y.~Harpaz, M.~Prasma, \emph{Strict versus coherent global equivariant homotopy theory}, In preparation.

\bibitem[Lur09]{Lur09} J.~Lurie, \emph{Higher topos theory}, No. 170. Princeton University Press, (2009).

\bibitem[Lur11]{Lur11} J.~Lurie, {Higher Algebra}, preprint, available at
\href{http://www.math.harvard.edu/~lurie/papers/higheralgebra.pdf}{Author's Homepage} (2011).


\bibitem[Mac]{Mac} S.~MacLane, \emph{Categories for the Working Mathematician}, Graduate Texts in Mathematics, Springer-Verlag (1971).


\bibitem[MM]{Ieke}
S.~MacLane, I.~Moerdijk, \emph{Sheaves in geometry and logic: A first introduction to topos theory}, Springer (1992).

\bibitem[Pra]{Pra} M.~Prasma, \emph{Segal group actions}, preprint arxiv:1311.4749. 

\bibitem[Qui]{Qui} D.~G.~Quillen, \emph{Homotopical Algebra}, Lecture Notes in Math., Vol. 43, Springer-Verlag, New York (1967).

\bibitem[Roi]{Ro} A.~Roig, \emph{Model category structures in bifibred categories}, Journal of Pure and Applied Algebra 95.2, p.203--223, (1994).

\bibitem[HS]{HS} A.~Hirschowitz and C.~Simpson, \emph{Descente pour les n-champs (Descent for n-stacks)}, preprint arXiv:980.7049 (1998).

\bibitem[Shi]{Shi} B.~Shipley, \emph{A convenient model category for commutative ring spectra}, Contemporary Mathematics 346, p. 473--484, (2004).

\bibitem[SS]{SS} S.~Schwede and B.~E.~Shipley, \emph{Algebras and modules in monoidal model categories},
Proc. London Math. Soc. (3) 80 , no. 2, 491--511 (2000).

\bibitem[Sta]{St} A.~E.~Stanculescu, \emph{Bifibrations and weak factorisation systems}, Applied Categorical Structures 20.1, p.19--30, (2012).

\bibitem[Str]{Str} R. Street, \emph{Fibrations in bicategories}, Cahiers de Topologie et Géométrie Différentielle Catégoriques 21.2 (1980), p.111--160.

\bibitem[Kro]{Kro}
Kro, Tore August, \emph{Model structure on operads in orthogonal spectra}, Homology, Homotopy and Applications 9.2, 2007, p. 397--412.

\bibitem[To\"{e}]{To} B.~To\"{e}n, \emph{Derived Hall algebras}, Duke Mathematical Journal 135.3, p.587--615, (2006).

\bibitem[Whi]{Wh} D.~White, \emph{Model Structures on Commutative Monoids in General Model Categories}, preprint arXiv:1403.6759 (2014).


\end{thebibliography}
\end{document}